\documentclass[11pt,twoside]{amsart}
\usepackage[a4paper]{geometry}
\usepackage{amstext}
\usepackage{amsthm}
\usepackage{amssymb}
\usepackage{stmaryrd}
\usepackage[unicode=true,pdfusetitle,
 bookmarks=true,bookmarksnumbered=false,bookmarksopen=false,
 breaklinks=false,pdfborder={0 0 1},backref=false,colorlinks=false]
 {hyperref}

\makeatletter

\newcommand{\lyxmathsym}[1]{\ifmmode\begingroup\def\b@ld{bold}
  \text{\ifx\math@version\b@ld\bfseries\fi#1}\endgroup\else#1\fi}

\numberwithin{equation}{section}
\numberwithin{figure}{section}
\theoremstyle{plain}
\newtheorem*{theorem*}{\protect\theoremname}
\newtheorem*{corollary*}{\protect\corollaryname}
\newtheorem*{lemma*}{\protect\lemmaname}
\theoremstyle{plain}
\newtheorem{theorem}{\protect\theoremname}
\theoremstyle{plain}
\newtheorem{lemma}[theorem]{\protect\lemmaname}
\theoremstyle{plain}
\newtheorem{proposition}[theorem]{\protect\propositionname}
\theoremstyle{remark}
\newtheorem{remark}[theorem]{\protect\remarkname}
\theoremstyle{plain}
\newtheorem{corollary}[theorem]{\protect\corollaryname}
\theoremstyle{definition}
\newtheorem{definition}[theorem]{\protect\definitionname}
\theoremstyle{definition}

\theoremstyle{plain}
\newtheorem{assumption}[theorem]{\protect\assumptionname}

\usepackage{amssymb,amsthm,amsmath,amsfonts,amscd}
\usepackage{graphicx}
\usepackage{url}
\usepackage{color}
\usepackage[active]{srcltx}
\usepackage[matrix,arrow]{xy}
\usepackage{mathrsfs}
\usepackage{enumerate}
\usepackage{amsopn} 
\usepackage{bbm} 
\usepackage[normalem]{ulem} 
\usepackage{cite}
\usepackage{hyperref}
\allowdisplaybreaks[1]
\usepackage[english]{babel}
\usepackage{bbm}
\usepackage{mhchem}
\usepackage{enumitem}
\date{\today}

\newcommand{\R}{\mathbb{R}} 
\newcommand{\E}{\mathcal{E}}
 
\newcommand{\N}{\mathbb{N}}

\makeatother

\providecommand{\corollaryname}{Corollary}
\providecommand{\definitionname}{Definition}
\providecommand{\examplename}{Example}
\providecommand{\lemmaname}{Lemma}
\providecommand{\remarkname}{Remark}
\providecommand{\theoremname}{Theorem}
\providecommand{\assumptionname}{Assumption}
\providecommand{\propositionname}{Proposition}

\begin{document}


\title[Pathwise McKean-Vlasov theory]{Pathwise McKean-Vlasov theory with additive noise}

	\begin{abstract}
	We take a pathwise approach to classical McKean-Vlasov stochastic differential equations with additive noise, as e.g.~exposed in Sznitmann \cite{sznitman1991topics}. Our study was prompted by some concrete problems in battery modelling \cite{guhlke2018stochastic}, and also
	 by recent progrss on rough-pathwise  McKean-Vlasov theory, notably Cass--Lyons \cite{MR3299600}, and then Bailleul, Catellier and Delarue \cite{bailleul2018mean}. Such a ``pathwise McKean-Vlasov theory'' can be traced back to Tanaka \cite{MR780770}. This paper can be seen as an attempt to advertize 
	   the ideas, power and simplicity of the pathwise appproach, not so easily extracted from \cite{bailleul2018mean, MR3299600, MR780770}, together with a number of novel applications. These 
	   include mean field convergence without a priori independence and exchangeability assumption; common noise, c\`adl\`ag noise, and reflecting boundaries. 
	 Last not least, we generalize  Dawson--G\"artner large deviations and the central limit theorem to a non-Brownian noise setting.
	
	%
	%
\end{abstract}


\author{Michele Coghi, Jean-Dominique Deuschel, Peter K. Friz, Mario Maurelli}

\address[M. Coghi]{Weierstrass Institute for Applied Stochastic, Mohrenstrasse 39, 10117 Berlin and Faculty of Mathematics,
	University of Bielefeld,
	33615 Bielefeld,
	Germany}
\email{michele.coghi@gmail.com}

\address[J.-D. Deuschel]{Technische Universit\"at Berlin, Str. des 17. Juni 136, 10587 Berlin, Germany}
\email{deuschel@math.tu-berlin.de}

\address[P.K. Friz]{Technische Universit\"at Berlin, Str. des 17. Juni 136, 10587 Berlin and Weierstrass Institute for Applied Stochastic, Mohrenstrasse 39, 10117 Berlin, Germany}
\email{friz@math.tu-berlin.de}

\address[M. Maurelli]{University of York and University of Edinburgh}
\email{mario.maurelli@york.ac.uk }

\maketitle
	
	\section{Introduction}
	\label{section:introduction}
	
	We consider the following \textit{generalized} McKean-Vlasov stochastic differential equation (SDE) on a probability space $(\Omega, \mathcal{A}, \mathbb{P})$, 
	\begin{equation}\label{intro:mckean vlasov equation}
	\left\{
	\begin{array}{l}
	dX_t = b(t, X_t, \mathcal{L}(X_t)) dt + dW_t\\
	X_0 = \zeta.
	\end{array}
	\right.	
	\end{equation}
	The input data to the problem is the random initial data and noise
	\begin{equation*}
	\begin{array}{cccc}
	(\zeta, W): & \Omega &\to &\mathbb{R}^d \times  C_T,
	\end{array}
	\end{equation*}
	and $$X: \Omega \to  C_T := C([0,T], \mathbb{R}^d)$$ is the solution (process). We denote by $\mathcal{L}(Y)$ the law of a random variable $Y$. Classically, one takes 
	$W$ as a Brownian Motion. For us, it will be crucial to avoid any a priori specification of the noise.  
	Indeed, we are not even asking for any filtration on the space $\Omega$ and equation \eqref{intro:mckean vlasov equation} will be studied pathwise.
	For $p \in [1, \infty)$, let $\mathcal{P}_p(\mathbb{R}^d)$ be the space of probability measures on $\mathbb{R}^d$ with finite $p$-moment endowed with the $p$-Wasserstein metric. The drift is a function
	\begin{equation*}
	b: [0,T] \times \mathbb{R}^d \times \mathcal{P}_p(\mathbb{R}^d) \to \mathbb{R}^d,
	\end{equation*}
	which 
	is assumed uniformly Lipschitz continuous in the last two variables, cf. Assumption \ref{assumption} below.
	
	\bigskip
	In a nutshell, McKean-Vlasov equations are SDEs which depend on the law of the solution. They have been extensively studied in the literature, for a comprehensive introduction we refer to \cite{sznitman1991topics}. They arise in many applications as limit of systems of interacting particles, for instance in the theory of mean field games developed by Lasry and Lions \cite{MR2269875, MR2271747, MR2295621}. Other interesting applications arise in fluid-dynamics \cite{bessaih2017mean, MR750980, MR3254330}, also with common noise features, neuroscience \cite{LucSta2014, Tou2014, DIRT2015} and macroeconomics \cite{nadtochiy2018mean}, also involving general driving signals. Last not least, our motivation 
comes from a recent battery model, cf. \eqref{intro:battery model} below, taken from \cite{guhlke2018stochastic}, which is of the form \eqref{intro:mckean vlasov equation} but with reflecting boundary as given in \eqref{intro:McKVlaSkor}.

	
	\medskip
	Closely related to the McKean-Vlasov equation is the
	system of particles (classically) driven by independent Brownian motions $W^i$, with independent identically distributed (i.i.d.) initial conditions $\zeta^i$,
	\begin{equation}\label{intro:interacting particles}
	\left\{
	\begin{array}{l}
	dX^{i, N}_t = b(t, X^{i, N}_t, L^N(X^{(N)}_t)) dt + dW^{i}_t\\
	X^{i, N}_0 = \zeta^i,
	\end{array}
	\right.	
	\quad i= 1, \dots, N.
	\end{equation}
The particles interact with each other through the empirical measure, which is defined as follows. Given a space $E$ (such as $\mathbb{R}^d$ or $C_T$) and a vector $x^{(N)} = (x^1,\dots,x^N) \in E^N$, we call $\mathcal{P}(E)$ the space of probability measures over $E$ and we define
\begin{equation*}
L^N(x^{(N)}) := \frac1N \sum_{i=1}^N \delta_{x^i} \in \mathcal{P}(E).
\end{equation*}
Let $X$ be a solution to equation \eqref{intro:mckean vlasov equation} with inputs $(\zeta, W)$ distributed as $(\zeta^1, W^1)$. When the number of particles, $N$, grows to infinity, we have the following a.s.~convergence in $\mathcal{P}(C_T)$ equipped with the usual weak-$*$ topology,
\begin{equation}
\label{intro:convergence empirical measure}
L^N(X^{(N)}(\omega)) \overset{\ast}{\rightharpoonup} \mathcal{L}(X), 
\quad \mbox{for }\mathbb{P}-\mbox{a.e. }\omega.
\end{equation} This result, as well as the well-posedness of equation \eqref{intro:mckean vlasov equation} is proved in \cite{sznitman1991topics} when the particles are exchangeable and subjects to independent inputs. This approach can be generalized to more general diffusion coefficients \cite{kurtz1999particle, MR1797090, CoghiGess}  using standard semi-martingale theory.

{\bf Rough paths:}		
	Cass and Lyons \cite{MR3299600} study McKean-Vlasov equations in the framework of rough paths. That is, they construct (rough) pathwise solutions to the McKean-Vlasov equation driven by suitable random rough paths, which lets them go beyond the classical case when $W$ is a semi-martingale under $\mathbb{P}$. They can treat the case multiplicative noise, that is with our $dW$ replaced by $\sigma(X) d\mathbf{W}$, but with mean field dependence only in the drift.
	This problem is revisited by Bailleul \cite{MR3420480} in the case of a Lipschitz dependence of $b$ on the measure. Finally, Bailleul et al. \cite{bailleul2018mean, bailleul2019propagation} study the general case when both $b$ and $\sigma$ are (Lipschitz) dependent on the law of the solution.  This requires extra assumptions of differentiability with respect to the measure argument.
	The rough path case is technically more involved, it especially requires more care when studying the mean-field convergence since the solution map ($\mathcal{L}(\zeta, W) 
	\mapsto \mathcal{L}(X)$) is continuous, but not Lipschitz (cf. \cite[Rmk 4.4.]{bailleul2019propagation}), in contrast our Lipschitz estimate in Theorem \ref{main theorem} below. For a different approach to rough differential equations with common noise, we refer also \cite{coghi2019rough}.

{\bf Tanaka:} As already mentioned, in the context of battery modelling with additive noise \cite{guhlke2018stochastic}, no rough path machinery is necessary, leave alone some formidable difficulties for rough differential equations to deal with reflecting boundaries \cite{aida2015reflected, deya2018one}. This was the initial motivation for our pathwise study, which soon turned out informative and rather pleasing in the generality displayed here. As our work neared completion we realized that we were not the first to go in this direction: the basic idea can be found (somewhat hidden) in a paper by Tanaka, \cite[Sec.2]{MR780770}. (There is no shortage of citations to \cite{MR780770}, but we are unaware of any particular work that makes use of the, for us, crucial Section 2 in that paper.) May that be as it is, advertising this aspect of Tanaka's work, as pathwise ancestor to \cite{MR3299600, MR3420480,bailleul2018mean}, is another goal of this note, and in any case there is no significant overlap of our results with \cite{MR780770}.

The main intuition of Tanaka \cite{MR780770} and subsequent works is that equation \eqref{intro:interacting particles} can be interpreted as equation \eqref{intro:mckean vlasov equation} by using a transformation of the probability space and the input data. We explain this connection between the equations in Section \ref{section:particle approximation}. This approach makes it possible to reduce the study of the mean field limit to a stability result for equation \eqref{intro:mckean vlasov equation}. This implies in particular that there is no need for asymptotical independence or exchangeability of the particles in order to obtain convergence \eqref{intro:convergence empirical measure}. Indeed, one can show that the solution map 
	\begin{equation*}
	\mathcal{L}(\zeta, W) 
	\mapsto \mathcal{L}(X)
	\end{equation*}
	that associates the law of the solution to the law of the inputs is continuous, and as soon as there is convergence for the law of the input data there is also convergence for the law of the solution. No independence, nor identical distributions (or even exchangeability) for the inputs are required, as we explain in Sections \ref{section:particle approximation} and \ref{section:common noise}. 
	%
		\subsection*{Main ideas} 
		Given a Polish space $E$, we work on the space of probability measures with finite $p$-th moment, $\mathcal{P}_p(E)$, endowed with the Wasserstein distance $\mathcal{W}_p$ (see Section \ref{section:notation} for the precise definition).
	The idea is to construct the solution map of equation \eqref{intro:mckean vlasov equation}, for a generic probability measure $\mu$,
	\begin{equation}
	\label{intro: map}
	\Phi : 
	\mathcal{P}_p(\R^d\times C_T) \times \mathcal{P}_p(C_T) \to \mathcal{P}_p(C_T),
	\qquad
	(\mathcal{L}(\zeta, W), \mu) \mapsto \mathcal{L}(X^\mu).
	\end{equation}
	Here $X^\mu$ is the pathwise solution to equation \eqref{intro:mckean vlasov equation} when the inputs are $\zeta, W$ and the measure in the drift is given as $\mu$, instead of the law of $X$. Existence and uniqueness of the solutions of the McKean-Vlasov equation \eqref{intro:mckean vlasov equation} follow as a fixed point argument of the parameter dependent map $\Phi$. Indeed, one can prove that, for fixed $(\zeta, W)$, the map $\Phi(\mathcal{L}(\zeta, W), \cdot)$ is a contraction on the space $\mathcal{P}_p(C_T)$. Hence, there is a unique fixed point $\bar \mu := \bar \mu(\mathcal{L}(\zeta, W)) = \Phi(\mathcal{L}(\zeta, W), \bar \mu)$. This fixed point uniquely determines a pathwise solution $X^{\bar \mu}$ to equation \eqref{intro:mckean vlasov equation}.
	
	Since $\Phi$ is Lipschitz continuous in all its arguments, it follows from Proposition \ref{contraction theorem} that also the map that associates the parameter to the fixed point, namely $\Psi$ defined in \eqref{eq:definition psi} is Lipschitz continuous. This is the stability result that we need in order to prove convergence of the particle system.
	
	
	\pagebreak 
	\subsection*{Main results} In this setting, we obtain the following list of results. 
		\begin{theorem*}[see Theorem \ref{main theorem}]
	Let $p \in [1, \infty)$ and assume $b$ Lipschitz. For $i=1,2$, let $(\zeta^i, W^i) \in L^p(\mathbb{R}^d \times C_T, \mathbb{P}^i)$ be two sets of input data. There exist unique pathwise solutions $X^i \in L^p(C_T)$ to equation \eqref{intro:mckean vlasov equation}, driven by the respective input data. Moreover,
	\begin{equation*}
	\mathcal{W}_{p}(\mathcal{L}(\zeta^1, W^1, X^1), \mathcal{L}(\zeta^2, W^2, X^2)) 
	\leq C \mathcal{W}_{p}(\mathcal{L}(\zeta^1, W^1), \mathcal{L}(\zeta^2, W^2)),
	\end{equation*}
	for some constant $ C =  C(p, T, b) >0$.
\end{theorem*}

We obtain a similar results for the case when the driver $W$ is a random variable over the c\`adl\`ag space $D_T$.

\begin{theorem*}[see Lemmas \ref{lem: jump: existence} and \ref{lem: jump continuity}]
	Assume $b$ Lipschitz and bounded. For every $(\zeta, W) : \Omega \to \mathbb{R}^d \times D_T$ measurable, there exists a unique pathwise solution $X : \Omega \to D_T$ to equation \eqref{intro:mckean vlasov equation} driven by $(\zeta, W)$. Moreover, the map
	\begin{equation*}
	\Phi : 
	\mathcal{P}(\R^d\times D_T) \to \mathcal{P}(D_T),
	\qquad
	\mathcal{L}(\zeta, W)\mapsto \mathcal{L}(X),
	\end{equation*}
	is continuous with respect to the weak topology.
\end{theorem*}
%

	We note that, in the case of jump processes, we have only weak continuity of the law of the solution with respect to the law of the inputs. We don't prove Lipschitz continuity with respect to the stronger Wasserstein norm $\mathcal{W}_p$.
	
As application off the main result, we have
	\begin{corollary*}[see Theorem \ref{classical mean field result}]
	        Consider the $N$-particle system (\ref{intro:interacting particles}) with (not necessarily Brownian! not necessarily independent!) random driving noise $W^{(N)} := (W^{1,N},\dots,W^{N,N})$ and initial data $\zeta^{(N)} := (\zeta^{1,N},\dots,\zeta^{N,N})$.  Assume convergence (in $p$-Wasserstein sense) of the empirical measure 
	        \begin{equation*}
	        L^N (\zeta^{(N)}(\omega),W^{(N)}(\omega)) \to \nu \in \mathcal{P}_p(\mathbb{R}^d \times C_T) 
	       \end{equation*}
	       for a.e.~$\omega$ (resp. in probability) w.r.t.~$\mathbb{P}$. Then the empirical measure $L^N(X^{(N)})$ of the particle system converges in the same sense and the limiting law 
	        is characterized by a generalized McKean-Vlasov equation, with input data distributed like $\nu$.
	\end{corollary*}
	 
	 Natural non-i.i.d.~situations arises in presence of common noise, cf.~Section \ref{section:common noise}, or in the presence of heterogeneous inputs, cf.~Section \ref{section:heterogeneous mean field}. In an i.i.d.~setting, the required assumption is (essentially trivially) verified by the law of large number. 
         Independent driving fractional Brownian motions, for instance, are immediately covered. Another consequence concerns the large deviations.
         \begin{definition}
         	\label{def: large deviations}
         	Let $E$ be a Polish space and $(\mu^N)_{N\in \N}$ a sequence of Borel probability measures on $E$. Let $(a_N)_{N\in\N}$ be a sequence of positive real numbers with $\lim_{N\to\infty}a_N = \infty$. Given a lower semicontinuous function $I:E\to[0,\infty]$, the sequence $\mu^N$ is said to satisfy a \textit{large deviations principle} with rate $I$ if, for each Borel measurable set $A \subset E$,
         	\begin{equation*}
         		- \inf_{x\in A^{\circ}} I(x) \leq \liminf a_N^{-1}\log(\mu^N(A)) \leq \limsup_{N\to\infty} a_N^{-1} \log(\mu^N(A)) \leq -\inf_{x\in \bar A}I(x).
         	\end{equation*}
         	Here $A^\circ$ is the interior of $A$ and $\bar A$ its closure. Moreover, if the sublevel sets of $I$ are compact, then $I$ is said to be a \textit{good} rate function.
         	
         	We say that a sequence of random variables $(X^N)_{N\in\N}$ on $E$ satisfies a large deviations principle, if the sequence of the distributions $(\mathcal{L}(X^N))_{N\in \N}$ does.
         \end{definition}
         
          The following generalizes a classical result of Dawson--G\"artner \cite{MR885876}, see also Deuschel et al. \cite{MR3804792}.
	\begin{corollary*}[see Theorem \ref{ldp:large deviations iid inputs}]
		In the i.i.d.~case, the empirical measure $L^N(X^{(N)})$ satisfies a large deviations principle with rate function, defined on a suitable Wasserstein space over $C_T$,
		\begin{equation*}
		\mu \mapsto H(\mu \mid \Phi( \mathcal{L}(\zeta, W) , \mu)), 
		\end{equation*}
		where $H$ is the relative entropy and $\Phi$ is introduced below.
	\end{corollary*}	
	This result is consistent with the one obtained in \cite[Theorem 5.1]{MR780770}, for the case of drivers given as i.i.d. Brownian motions.
	
	One can easily drop the i.i.d.~assumption, and replace $H$ by an ``assumed'' large deviations principle $I$ for the convergence of the input laws. In this case the outputs satisfy a large deviations principle. 
		\begin{corollary*}[see Lemma \ref{ldp:contraction principle}]
		If the empirical measure of the inputs $L^N(\zeta^{(N)}, W^{(N)})$ satisfies a large deviations principle with (good) rate function $I$, then the empirical measure $L^N(X^{(N)})$ satisfies a large deviations principle with (good) rate function $\mu \mapsto I(f^\mu_{\#}\mu)$, defined on a suitable Wasserstein space over $C_T$. Here $f^\mu$ is defined in \eqref{eq:definion f mu}.
	\end{corollary*}	
	Think of $f^\mu$ as the function that reconstruct the inputs (initial condition, driving path) from the solution of an ordinary differential equation (ODE).
	
	Moreover, we study the fluctuations of the empirical measure. We can prove the following central limit theorem type of result
	\begin{corollary*}[see Corollary \ref{cor: central limit thm}]
		Let $\varphi$ be a test function. Assume that the drift $b$ is differentiable in both the spacial and the measure variable (see Assumption \ref{asm: drift clt}).
		The following converges in distribution to a Gaussian random variable, as $N\to \infty$,
		$$
		Y^N := \sqrt{N} \left( \frac{1}{N} \sum_{i=1}^{N} \varphi( X^{ i , N } ) - \mu ( \varphi ) \right) .
		$$
	\end{corollary*}
	
	The method presented here can be also applied to SDE defined in a domain $D \subset \mathbb{R}^d$, assumed to be a convex polyhedron for simplicity, and with reflection at the boundary. We consider the generalized McKean-Vlasov Skorokhod problem
	\begin{equation}
		\label{intro:McKVlaSkor}
		\left\{
		\begin{array}{l}
		dX_t = b(t,X_t,\mathcal{L}(X_t,k_t)) dt + dW_t -dk_t, \quad X_0=\zeta,\\
		d|k|_t = 1_{X_t\in \partial D} d|k|_t,\ \ dk_t = n(X_t) d|k|_t.
		\end{array} 
		\right.
	\end{equation}
	We have the following:

\begin{theorem*}[see Theorem \ref{thm:wellpos_cont_boundary}]
	Let $p \in [1, \infty)$ and assume $b$ Lipschitz. For $i=1,2$, let $(\zeta^i, W^i) \in L^p(\bar D \times C_T, \mathbb{P}^i)$ be two sets of input data. Then there exist unique pathwise solutions $(X^i,k^i)$ to the generalized McKean-Vlasov Skorokhod problem \eqref{intro:McKVlaSkor}, driven by the respective input data. Moreover,
		\begin{align*}
		\mathcal{W}_{p}(\mathcal{L}(\zeta^1, W^1, X^1, k^1), \mathcal{L}(\zeta^2, W^2, X^2, k^2)) 
		\leq C \mathcal{W}_{p}(\mathcal{L}(\zeta^1, W^1), \mathcal{L}(\zeta^2, W^2)).
		\end{align*}
		with $ C =  C(p, T, b) > 0$.
\end{theorem*}

	\subsection*{Battery modelling} Our initial motivation for the heterogeneous particles case comes from modeling lithium-ion batteries. The numerical simulations of \cite{guhlke2018stochastic} indicate that the capacity of the battery and its efficiency is mainly determined by the size distribution of the lithium iron phosphate particles. It is thus important to allow for the particles to be of fixed different, predetermined sizes. 
	
	Lithium-ion batteries are the most promising storage devices to store and convert chemical energy into electrical energy and vice versa. In \cite{guhlke2018stochastic} lithium-ion batteries are studied where at least one of the two electrodes stores lithium within a many-particle ensemble, for example each particle of the electrode is made of Lithium-iron-phosphate. One of the practical achievements of \cite{guhlke2018stochastic} consists of the conclusions that the capacity of the battery and its efficiency as well is dominantly determined by the size distribution of the storage particles, ranging from $20$ to $1000$ nanometers. 
	The radii $r^i$ of the particles in the battery are distributed according to a distribution $\lambda \in \mathcal{P}([20, 1000])$. However, in the numerical simulations, it leads to better accuracy to artificially choose the radii in advance, instead of randomly sample them. For instance, assume that we want to simulate $1000$ particles, whose radii can be of exactly two given sizes, $r_1, r_2$, with equal probability. It is much more convenient to choose $500$ particles of radius $r_1$ and $500$ of radius $r_2$, instead of sampling them from a binomial, as this could lead to imbalanced simulations and introduce an extra source of error. For this reason it is important that the theoretical results support the use of carefully chosen radii $r^i$ of different length, such their empirical measure converges, as the number of particles grows, to a desired distribution $\lambda$. The radii so chosen, are deterministic (hence, independent), but not identically distributed.
	
	The dynamics of the charging/discharging process is modeled in \cite{guhlke2018stochastic} by a coupled system of SDEs for the evolution of the lithium mole fractions $Y^{i,N} \in [0,1]$ of particles $i = 1, 2, \dots, N$ of the particle ensemble.
	The evolution of $Y^{i,N}$ over a time interval $[0,T]$ is described by the following system of SDEs
	\begin{equation}
	\label{intro:battery model}
	\left\{
	\begin{array}{l}
	dY^{i,N}_t = \frac{1}{\tau_i} (\Lambda_{t,Li} - \mu_{Li}(Y^{i,N}_t)) dt + \sigma_i dW_t^i - dk^{i,N}_{t} \\
	d|k^{i,N}_{t} | = 1_{ Y^{i,N}_t \in \{ 0 , 1 \} } d | k^{i,N}_{t} |, \quad dk^{i,N}_{t} = n(Y^{i,N}_{t}) d | k^{i,N}_{t} |,\\
	Y^{i,N}_0 = a \in [0, 1].
	\end{array}
	\right.
	\quad i=1, \dots, N,
	\end{equation}
	We assume that all the particle have the same amount of lithium mole fraction $a \in [0,1]$ at time $t=0$. In practice, this initial condition is very close to $1$, when the battery is empty and very close to $0$, when the battery is charged.
	The particles are driven by a family of independent Brownian motions $W^{(N)}:= (W^i)_{1\leq i \leq N}$, which account for random fluctuations that can occur within the system during charging and discharging. The quantity $\tau_i \equiv \tau(r_i)$, which is related to the relaxation time and to the particle active surface, is a function of the radius $r_i$ of the particle. As discussed earlier, the radii can only have values in a fixed range $I := [r_{min}, r_{max}] \subset (0,\infty)$. We assume that $\tau:I\to \R$ and $\tau^{-1} : I \to \R$ is Lipschitz and bounded. We also assume that $\sigma_i=\sigma(r_i)$ for a Lipschitz and bounded function $\sigma:I \to \R$. The term $\mu_{Li} : \R \to \R$ is the chemical potential of the Lithium and, in this framework, it is also taken Lipschitz and bounded. The interaction between particles is encoded in the surface chemical potential 
	$$
	\Lambda_{t,Li} := \frac{\sum^N_{j=1} \left[ V_j\dot{q}_t +\mu_{Li}(x)\frac{V_j}{\tau^j} \right]}{\int \frac{V_j}{\tau_j}},
	$$
	where $q_t$ is a given $C^1$ function characterizing the state of charge of the battery at time $t \in [0,T]$ and $V_j$ is the volume of the $j$-th particle.
	By the assumptions on $\mu_{Li}$ and $\tau$ and the bounds on the radii, the surface chemical potential is a bounded and Lipschitz continuous function of the empirical distribution of the Lithium mole fractions and radii, $\mu^N := \frac{1}{N}\sum_{i=1}^{N} \delta_{(Y^i_t,r^i)}$.
	
	Moreover, we impose on the particles Skorokhod-type boundary conditions, of the same type as the ones described in Section \ref{section:boundary conditions}. We call $n(x)$ the outer unit normal vector, which, in this case, reduces to $n(x) = (-1)^{x+1}$, for $x\in \{0,1\}$. This will force the mole fraction of each particle to remain in $[0,1]$. Reflecting boundary conditions are imposed also in \cite{DHMRW15}, which considers the PDE counterpart of the model \eqref{intro:battery model} here in the case of $\tau_j$ independent of $j$. Those boundary condtions are similar but not identical to the boundary conditions here: in \cite{DHMRW15}, the surface chemical potential $\Lambda_{t,Li}$ accounts also for mean field interactions from the boundaries, we disgard those interactions here.
		
	Under the previous assumptions, the particle system \eqref{intro:battery model} can be essentially treated combining the results of Theorem \ref{thm:wellpos_cont_boundary} and Corollary \ref{cor: heterogeneus particles}, as follows. To unify the notation to the rest of the paper, we define 
	$$
	b: [0,T] \times \R \times \R \times \mathcal{P}(\R \times \R) \to \R,
	\qquad
	(t,x,r, \nu) \mapsto \frac{1}{\tau(r)} (-\mu_{Li}(x) +\Lambda{t,Li}),
	$$
	where (calling $V(r)=4\pi r^3/3$ the volume of the particles of radius $r$),
	$$
	\Lambda_{t,Li} = \frac{\int \left[ V(r)\dot{q}_t +\mu_{Li}(x)\frac{V(r)}{\tau(r)} \right] \nu(d(x,r))}{\int \frac{V(r)}{\tau(r)} \nu(d(x,r))},
	$$
	and we consider the following \emph{generalized} McKean-Vlasov Skorokhod equation
	\begin{equation}
	\label{intro: mckean-vlasov battery}
	\left\{
	\begin{array}{l}
	dX_t = b(t,X_t, R , \mathcal{L}(X_t, R)) dt + \sigma(R)dW_t - dk_t, \quad X_0=a,\\
	d|k|_t = 1_{X_t\in \partial D} d|k|_t,\ \ dk_t = n(X_t) d|k|_t,
	\end{array} 
	\right.
	\end{equation}
	The input data are given by $a\in [0,1]$, $R \in L^p(I)$ and $W \in L^p(C_T)$, for $p\in [1,\infty)$. The solution is a couple $(X, k) \in C_T([0,1]) \times C_T$. When $W := W^{(N)} \in L^p(\Omega_N, C_T)$ and $R = R^{(N)} = (r^1, \cdots , r^N) \in L^p(\Omega_N, C_T(I))$ (the radii are constant path in $I$), we recover the system \eqref{intro:battery model}. We assume that the radii $r^i$ are sampled from a distribution $\lambda \in \mathcal{P}_p(C_T(I))$ in such a way that $L^N(R^{(N)}) \overset{\ast}{\rightharpoonup} \lambda$, this gives the limit process $(X,k)$ solution to \eqref{intro: mckean-vlasov battery}, driven by $(W, R) \in C_T \times C_T(I)$ with law $\mu_{\mathcal{W}} \otimes \lambda$ (here $\mu_{\mathcal{W}}$ is the Wiener measure). We summarize this in the following proposition.
	
	\begin{proposition}
		Let $p\in [1,\infty)$ and let $(W^i)_{i\in \N}$ be a family of independent Brownian motions on $\R$. Assume $I \subset (0,\infty)$ is a closed interval and let $(r^i)_{i\in\N} \subset I$ be a sequence in $I$, such that
		$$
		L^N(R^{(N)}) \overset{\ast}{\rightharpoonup} \lambda \in \mathcal{P}_p(I).
		$$
		Then, for every $N \in \N$, equation \eqref{intro:battery model} admits a unique solution $(Y^{(N)}, k^{(N)}) := (Y^{i,N}, k^{i,N})_{i=1, \dots , N}$. Moreover, 
		$$
		L^{(N)}(Y^{(N)}, k^{(N)}) \overset{\ast}{\rightharpoonup} \mathcal{L}(X, k),
		$$
		where $(X,k)$ is a solution to equation \eqref{intro: mckean-vlasov battery}, with input data $(W, R) \in C_T \times C_T(I)$ with law $\mu_{\mathcal{W}} \otimes \lambda$.
	\end{proposition}

	The proof of this proposition (which we will not give in full details) follows exactly as the proof of Corollary \ref{cor: heterogeneus particles} and Remark \ref{rmk:heterogeneous_noise}, with the difference that, instead of Theorem \ref{main theorem}, one applies Theorem \ref{thm:wellpos_cont_boundary}.

	\subsection*{Structure of the paper}
	In Section \ref{section:main result} we prove the well-posedness for the \textit{generalized} McKean-Vlasov equation \eqref{intro:mckean vlasov equation}. In Section \ref{section:applications} we present applications to classical mean field particle approximation, heterogeneous mean field and mean field with common noise as corollaries of the main result.
	Then, we study other (classical) asymptotic for the particles as a straightforward applications: a large deviations result in Section \ref{section:large deviations}; a central limit theorem in Section \ref{sec: central limit theorem}.
	Finally, we adapt the result to study McKean-Vlasov equations with reflection at the boundary, see Section \ref{section:boundary conditions}.

	\subsection{Notation}
	\label{section:notation}
	
	Given $p$ in $[1,+\infty)$ and a Polish space $E$, with metric induced by a norm $\Vert \cdot \Vert_{E}$, we denote by $\mathcal{P}_p(E)$ the space of probability measures on $E$ with finite $p$-moment, namely the measures $\mu$ such that
	\begin{equation*}
	\int_{E} \Vert x \Vert^p_{E} d\mu(x) < +\infty.
	\end{equation*}
	
	For $T>0$, we denote by $C_T(\mathbb{R}^d) := C([0,T], \mathbb{R}^d)$ (the space of continuous functions from $[0,T]$ to $\mathbb{R}^d$), endowed with the supremum norm $\Vert f \Vert_{\infty:T} := \sup_{t \in [0,T]} \vert f(t) \vert$, for $f \in C_T(\mathbb{R}^d)$. When there is no risk of confusion about the codomain, we denote the space of continuous functions by $C_T$. Moreover, when there is non risk of confusion about the time interval, we use the lighter notation $\Vert \cdot \Vert_{\infty}$. Moreover, we call $C_{T,0} = \{ \gamma \in C_T \mid \gamma_0 = 0 \}$, the subsets of paths that vanish at time $0$.
	
	For a domain $\bar{D}$ in $\mathbb{R}^d$, we denote by $C_T(\bar{D}):= C([0,T], \bar{D})$ (continuous functions from $[0,T]$ to $\bar{D}$), endowed with the supremum norm $\Vert \cdot \Vert_{\infty}$.
	
	Given $t \in [0, T]$, the projection $\pi_t$ is defined as the function $\pi_t: C_T \to \mathbb{R}^d$ as $\pi_t(\gamma) := \gamma(t)$. We define the marginal at time $t$ of $\mu\in \mathcal{P}_p(C_T)$ as $\mu_t := (\pi_t)_{\#}\mu \in \mathcal{P}_p(\mathbb{R}^d)$. We also denote by $\mu\vert_{[0,t]}$ the push forward of $\mu$ with respect to the restriction on the subinterval $[0,t]$.
	
	Given a Polish space $(E, d)$, the $p$-Wasserstein metric on $\mathcal{P}_p(E)$ is defined as
	\begin{equation}\label{Wasserstein}
	\mathcal{W}_{E,p}(\mu, \nu)^p = \inf_{m \in \Gamma(\mu, \nu)} \iint_{E\times E} d(x, y)^p m(dx, dy),
	\quad \mu, \nu \in \mathcal{P}_p(E),
	\end{equation}
	where $\Gamma(\mu, \nu)$ is the space of probability measures on $E\times E$ with first marginal equal to $\mu$ and second marginal equal to $\nu$. We will omit the space $E$ from the notation when there is no confusion.
	
	We denote by $\mathcal{L}(X)$ the law of a random variable $X$.
	
	We use $C_{p}$ to denote constants depending only on $p$.

	Let $C_c^\infty$ and $C^n$ be the set of infinitely differentiable differentiable real-valued functions of compact support defined on $\mathbb{R}^d$ and the set of $n$ times continuously differentiable functions on $\mathbb{R}^d$ such that
	\begin{equation*}
	\Vert \varphi \Vert_{C^n} := \sum_{\vert \alpha \vert \leq n} \sup_{x\in\mathbb{R}^d} \vert D^\alpha \varphi \vert < +\infty.
	\end{equation*}
	
	Let $\operatorname{Lip}_1$ be the space of Lipschitz continuous functions in $C^0$, such that
	\begin{equation*}
	\Vert \varphi \Vert_{C^0}, \sup_{x \neq y \in \mathbb{R}^d} \frac{\vert \varphi(x) - \varphi(y) \vert}{\vert x - y \vert} \leq 1.
	\end{equation*}
	
	For $T>0$, we denote by $D_T(\mathbb{R}^d) := D([0,T], \mathbb{R}^d)$, the space of c\`adl\`ag functions (right-continuous with left limit) from $[0,T]$ to $\mathbb{R}^d$. When there is no risk of confusion about the codomain, we denote the space of cadlag functions by $D_T$.
	For $\gamma \in D_T(\mathbb{R}^d)$,  $\Vert \gamma \Vert_{\infty:T} := \sup_{t \in [0,T]} \vert f(t) \vert$. Moreover, when there is no risk of confusion about the time interval, we use the lighter notation $\Vert \cdot \Vert_{\infty}$.
	We endow $D_T$ with the Skohorod metric, defined as follows
	\begin{equation}
	\label{def: sko metric}
	\sigma(\gamma, \gamma^\prime) = \inf \left\{
	\lambda \in \Lambda \mid
	\Vert \lambda \Vert + \Vert \gamma - \gamma^\prime \circ \lambda \Vert_{\infty}
	\right\},
	\qquad
	\gamma, \gamma^\prime \in D_T,
	\end{equation}
	where $\Lambda$ is the space of strictly increasing bijections on $[0,T]$ and
	\begin{equation*}
	\Vert \lambda \Vert := \sup_{s\neq t} \left \vert \log \left( \frac{\lambda_{s,t}}{t-s} \right) \right \vert,
	\quad \lambda \in \Lambda.
	\end{equation*}
	
	The space $(D_T, \sigma)$ is a Polish space.

	\subsection*{Acknowledgements}
	Support from the Einstein Center Berlin, ECMath Project ``Stochastic methods for the analysis of lithium-ion batteries'' is gratefully acknowledged.
	J.-D. Deuschel and P.K. Friz are partially supported by DFG research unit FOR 2402. This project has received funding from the European Research Council (ERC) under the European Union's Horizon 2020 research and innovation programme (grant agreement No. 683164, PI P.K.Friz.) 
		The authors thank Nikolai Bobenko for carefully reading the paper and pointing out a mistake in the proof of Lemma \ref{weak cont:existence} in an earlier version.
	
	\section{The main result}
	\label{section:main result}
	
	In this section we study the \textit{generalized} McKean-Vlasov SDE on a probability space $(\Omega, \mathcal{A}, \mathbb{P})$, 
	\begin{equation}\label{mckean vlasov equation}
	\left\{
	\begin{array}{l}
	dX_t = b(t, X_t, \mathcal{L}(X_t)) dt + dW_t\\
	X_0 = \zeta.
	\end{array}
	\right.	
	\end{equation}
	Here the drift $b:[0,T]\times \mathbb{R}^d \times \mathcal{P}_p(\mathbb{R}^d)\rightarrow \mathbb{R}^d$ is a given Borel function, the input to the problem is the random variable
	\begin{equation*}
	\begin{array}{cccc}
	(\zeta, W): & \Omega &\to &\mathbb{R}^d \times C_T,
	\end{array}
	\end{equation*}
	and $X: \Omega \to  C_T$ is the solution. As we will see later, the law $\mathcal{L}(X)$ of the solution depends only on the law $\mathcal{L}(\zeta,W)$, for this reason we refer also to $\mathcal{L}(\zeta,W)$ as input.
	
	Note two differences here with respect to classical SDEs: the drift depends on the solution $X$ also through its law and $W$ is merely a random continuous paths; in particular, it does not have to be a Brownian motion. For these differences, it is worth giving the precise definition of solution.
	
	\begin{definition}\label{def solution}
		Let $(\Omega,\mathcal{A},\mathbb{P})$ be a probability space and let $\zeta:\Omega\rightarrow \mathbb{R}^d$, $W:\Omega\rightarrow C_T$ be random variables on it. A solution to equation \eqref{mckean vlasov equation} with input $(\zeta,W)$ is a random variable $X: \Omega \to C_T$ such that, for a.e.~$\omega$, the function $X(\omega)$ satisfies the following integral equality		
		\begin{equation*}
		X_t (\omega) = \zeta(\omega) + \int_{0}^{t}b(s, X_s(\omega), \mathcal{L}(X_s)) ds + W_t(\omega).
		\end{equation*}
	\end{definition}
	
	We assume the following conditions on $b$:
	\begin{assumption}
		\label{assumption}
		Let $p\in[1,\infty)$. The drift $b: [0,T] \times \mathbb{R}^d \times \mathcal{P}_p(\mathbb{R}^d) \to \mathbb{R}^d$ is a measurable function and there exists a constant $K_b$ such that, 
		\begin{equation*}
		\vert b(t, x, \mu) - b(t, x^\prime, \mu^\prime) \vert^p
		\leq K_b \left(\vert x - x^\prime \vert^p  
		+ \mathcal{W}_{\mathbb{R}^d,p}(\mu, \mu^\prime)^p\right),
		\end{equation*}
		$\forall t \in [0,T], x, x^\prime \in \mathbb{R}^d, \mu, \mu^\prime \in \mathcal{P}_p(\mathbb{R}^d)$.
	\end{assumption}
	
	Before giving the main result, we introduce some notation. For a given $\mu$ in $\mathcal{P}_p(C_T)$, we consider the SDE
	\begin{equation}\label{linear equation}
	\left\{
	\begin{array}{l}
	dY_t^\mu = b(t,Y_t^\mu, \mu_t) dt + dW_t\\
	Y^\mu_0 = \zeta.
	\end{array}
	\right.	
	\end{equation}
	We have the following well-posedness result
	\begin{lemma}
		\label{lem:solution linear}
		Under Assumption \ref{assumption}, for every input $(\zeta,W) \in L^p(\mathbb{R}^d \times C_T)$ and $\mu \in \mathcal{P}_p(C_T)$, there exists a unique $Y^\mu \in L^p(C_T)$ which satisfies, $\forall \omega \in \Omega$,
		\begin{equation*}
		Y^\mu_t(\omega) = \zeta(\omega) + \int_0^t b(s, Y_s^\mu(\omega), \mu_s)ds + W_t(\omega).
		\end{equation*}
		Moreover, denote by
		\begin{equation}
		\label{defn:solution map}
		\begin{array}{cccc}
			S^\mu: & \mathbb{R}^d\times C_T & \rightarrow  & C_T\\
			& (x_0, \gamma) & \mapsto & S^\mu(x_0, \gamma),
		\end{array}
		\end{equation}
		where $S^\mu(x_0, \gamma)$ is a solution to the ODE
		\begin{equation}
		\label{eq:ode}
		x_t = x_0 + \int_0^t b(s, x_s, \mu_s) ds + \gamma_t.
		\end{equation}
		Then, $Y^\mu = S^\mu(\zeta, W)$.
	\end{lemma}
	\begin{proof}
		For every couple $(x_0, \gamma) \in \mathbb{R}^d \times C_T$ the ODE \eqref{eq:ode} classically admits a solution $S^\mu(x_0,\gamma)$, which is continuous with respect to the inputs $(x_0, \gamma)$. It is easy to verify that $S^\mu(\zeta, W)$ solves equation \eqref{linear equation}.
		We only verify that $Y^\mu$ has finite $p$-moments. There exists a constant $C(p, b, T)$ such that
		\begin{equation*}
			\mathbb{E}\Vert Y^\mu \Vert_{\infty}^p 
			\leq \mathbb{E}\vert \zeta \vert^p 
			+ C\left( 
			1+
			\int_0^T\mathbb{E}\sup_{s\in[0,t]}\vert Y^\mu \vert_\infty^p dt
			+ \int_0^T\int_{\mathbb{R}^d} \vert x \vert^p d\mu_t(x)dt 
			\right)
			+ \mathbb{E}\Vert W \Vert_{\infty}^p.
		\end{equation*}
		We notice that $\int_{\mathbb{R}^d} \vert x \vert^p \mu_t(dx) \leq \int_{C_T} \Vert \gamma \Vert_{\infty}^p d\mu(\gamma) < +\infty$. Gronwall's inequality and the assumptions on $(\zeta, W)$ conclude the proof.
	\end{proof}
	We call
	\begin{equation}
	\label{eq:sol_map_law}
	\begin{array}{cccc}
	\Phi: &\mathcal{P}_p(\mathbb{R}^d\times C_T)\times \mathcal{P}_p(C_T) & \rightarrow & \mathcal{P}_p(C_T)\\
	& (\mathcal{L}(\zeta, W), \mu) & \mapsto & \mathcal{L}(Y^\mu) = (S^\mu)_\# \mathcal{L}(\zeta, W),
	\end{array}
	\end{equation}
	the push forward of a probability measure $\mathcal{L}(\zeta, W)$ under the solution map $S^\mu$ defined in \eqref{defn:solution map}. 
	
	Note that $X$ uniquely solves the McKean-Vlasov equation \eqref{mckean vlasov equation} with input $(\zeta, W)$, if and only if $\mathcal{L}(X)$ is a fixed point of $\Phi({\mathcal{L}(\zeta, W)}, \cdot )$:
	\begin{itemize}
		\item if $X$ solves \eqref{mckean vlasov equation}, then, by uniqueness for fixed $\mu=\mathcal{L}(X)$, $X= S^{\mathcal{L}(X)}(\zeta,W)$ $\mathbb{P}$-a.s.~and so $\mathcal{L}(X)$ is a fixed point of $\Phi({\mathcal{L}(\zeta, W)}, \cdot )$;
		\item conversely, if $\mu^{\mathcal{L}(\zeta, W)}$ is a fixed point of $\Phi({\mathcal{L}(\zeta, W)}, \cdot )$, then $X=S^{\mu^{\mathcal{L}(\zeta, W)}}(\zeta,W)$ has finite $p$-moment and solves \eqref{mckean vlasov equation}.
	\end{itemize}
	Hence existence and uniqueness for \eqref{mckean vlasov equation} in Theorem \ref{main theorem} follow from existence and uniqueness for fixed points of $\Phi^{\mathcal{L}(\zeta, W)}$, for any law $\mathcal{L}(\zeta, W)$.
	
	For this reason, the main ingredient in the proof of Theorem \ref{main theorem} is the following general proposition, a version of the contraction principle with parameters. The proof is postponed to the appendix.
	\begin{proposition}\label{contraction theorem}
		Let $(E, d_E)$ and $(F, d_F)$ be two complete metric spaces. Consider a function $\Phi : F \times E \to E$ with the following properties:
		\begin{itemize}
			\item[1)] (uniform Lipschitz continuity) there exists $L > 0$ such that
			\begin{equation*}
			d_E(\Phi(Q, P), \Phi(Q^\prime, P^\prime)) \leq L\left[d_E(P, P^\prime) + d_F(Q, Q^\prime)\right].
			\end{equation*}
			\item[2)] (contraction) There exist a constant $0< c <1$ and a natural number $k\in \mathbb{N}$ such that 
			\begin{equation*}
			\quad d_E((\Phi^Q)^k(P), (\Phi^Q)^k(P^\prime)) \leq c d_E(P, P^\prime) \ \ \forall Q \in F,\ \forall P,P^\prime\in E,
			\end{equation*}
			with $\Phi^Q(P) := \Phi(Q,P)$.
		\end{itemize}
		
		Then for every $Q \in F$ there exists a unique $P_Q \in E$ such that
		\begin{equation*}
		\Phi(Q, P_Q) = P_Q.
		\end{equation*}
		Moreover, 
		\begin{equation}\label{inequality}
		\forall Q, Q^\prime \in F, \quad d_E(P_Q, P_{Q^\prime}) \leq \tilde C d_F(Q, Q^\prime),
		\end{equation}
		where $\tilde C := \left(\sum_{i=1}^k L^i\right)(1-c)^{-1}$.
	\end{proposition}
	
	We give now the main result, from which most of the applications follow. It states well-posedness of the generalized McKean-Vlasov equation and Lipschitz continuity with respect to the driving signal.
	
	\begin{theorem}\label{main theorem}
		Let $T>0$ be fixed and let $p \in [1,\infty)$, assume Assumption \ref{assumption}.
		\begin{enumerate}[label=(\roman*), ref= \ref{main theorem} (\roman*)]
			\item For every input $(\zeta,W) \in L^p(\mathbb{R}^d\times C_T)$, the map $\Phi^{\mathcal{L}(\zeta, W)}$ has a unique fixed point, $\mu^{\mathcal{L}(\zeta, W)}$.
			\item \label{item 2 of main} The map that associates the law of the inputs to the fixed point, namely
			\begin{equation}
			\label{eq:definition psi}
			\begin{array}{cccc}
			\Psi : & \mathcal{P}_p(\mathbb{R}^d \times C_T) & \to & \mathcal{P}_p(C_T)\\
			& \nu &\mapsto &\mu^{\nu}
			\end{array}
			\end{equation}
			is well-defined and Lipschitz continuous.
			\item For every input $(\zeta,W)$, there exists a unique solution $X$ to the generalized McKean-Vlasov \eqref{mckean vlasov equation}, given by $X = S^{\Psi(\mathcal{L}(\zeta, W))}(\zeta, W)$.
			\item There exists a constant $\tilde C = \tilde C(p, T, b) > 0$ such that: for every two inputs $(\zeta^i,W^i)$, $i=1,2$ (defined possibly on different probability spaces) with finite $p$-moments, the following is satisfied
			\begin{align*}
			\mathcal{W}_{C_T, p}(\mathcal{L}(X^1), \mathcal{L}(X^2)) 
			\leq \tilde C \mathcal{W}_{\mathbb{R}^d \times C_T,p}(\mathcal{L}(\zeta^1, W^1), \mathcal{L}(\zeta^2, W^2)).
			\end{align*}
			In particular, the law of a solution $X$ depends only on the law of $(\zeta,W)$.
		\end{enumerate}
	\end{theorem}
	
	\begin{proof}
		The result follows from Proposition \ref{contraction theorem}, applied to the spaces $E := \mathcal{P}_p(C_T)$, $F:= \mathcal{P}_p(\mathbb{R}^d \times C_T)$ and the map $\Phi$ defined in \eqref{eq:sol_map_law}, provided we verify conditions $1)$ and $2)$.
		
		Let now $\mu \in E$ be fixed, let $\nu^1$ and $\nu^2$ be in $\mathcal{P}_p(\mathbb{R}^d\times C_T)$ and let $m$ be an optimal plan on $(\mathbb{R}^d\times C_T)^2$ for these two measures. We call optimal plan a measure $m$ that satisfies the minimum in the Wasserstein distance, see \eqref{minimum wasserstein}. 	On the probability space $((\mathbb{R}^d\times C_T)^2, m)$, we call $\zeta^i$, $W^i$ the r.v.~defined by the canonical projections and $Y^i=S^\mu(\zeta^i,W^i)$ the solution to equation \eqref{linear equation} with input $(\zeta^i, W^i)$. By definition of the Wasserstein metric, we have that
		\begin{equation*}
		\mathcal{W}_{C_T, p}(\Phi(\nu^1,\mu), \Phi(\nu^2,\mu))^p = \mathcal{W}_{C_T, p}(\mathcal{L}(Y^1), \mathcal{L}(Y^2))^p
		\leq C_p\mathbb{E}_m\Vert Y^1 - Y^2\Vert_{\infty:T}^p.
		\end{equation*}
		The right hand side can be estimated using the equation,
		\begin{align*}
		\mathbb{E}_m\Vert Y^1 - Y^2\Vert_{\infty:T}^p
		\leq & C_p\mathbb{E}_m\vert \zeta^1 - \zeta^2\vert^p
		+ C_p\mathbb{E}_m\Vert W^1 - W^2\Vert_{\infty:T}^p \\
		& + K_bC_p\int_{0}^{T} \mathbb{E}_m \Vert Y^1 - Y^2\Vert_{\infty:t}^p dt.
		\end{align*}
		Using Gronwall's inequality we obtain
		\begin{align}
		\mathcal{W}_{C_T , p} (\mathcal{L}(Y^1), \mathcal{L}(Y^2))^p\nonumber
		\leq & C_p e^{TK_bC_p} \left( 
		\mathbb{E}_m\vert \zeta^1 - \zeta^2\vert^p
		+ \mathbb{E}_m\Vert W^1 - W^2\Vert_{\infty:T}^p
		\right)\nonumber\\
		= & \tilde L \mathcal{W}_{\mathbb{R}^d \times C_T, p}(\nu^1, \nu^2)^p, \label{estimate one of 1}
		\end{align}
		where $\tilde L := C_p e^{TK_bC_p}$.
		
		Let now $(\zeta, W)$ be fixed with law $\nu := \mathcal{L}(\zeta, W)$. Consider $\mu^1, \mu^2 \in E$ and call $S^{\mu^i}$, for $i=1,2$, the corresponding solution map as defined in \eqref{defn:solution map} (driven by the initial datum $\zeta$ and the path $W$). Let $t\in [0,T]$ be fixed. Using equation \eqref{linear equation} again, we get that
		\begin{align*}
		\int_{\mathbb{R}^d \times C_T} & \Vert S^{\mu^1}(x_0, \gamma) - S^{\mu^2}(x_0, \gamma)\Vert_{\infty:t}^p \; d\nu(x_0, \gamma)
		\leq K_pC_p \int_{0}^{t} \mathcal{W}_{C_s,p}(\mu^1\vert_{[0,s]}, \mu^2\vert_{[0,s]})^p ds \\
		& + K_p C_p \int_0^t \int_{\mathbb{R}^d \times C_T} \Vert S^{\mu^1}(x_0, \gamma) - S^{\mu^2}(x_0, \gamma)\Vert_{\infty:s}^p \;d\nu(x_0, \gamma) ds.
		\end{align*}
		We deduce by the definition of $\Phi^\nu := \Phi(\nu, \cdot)$ and Wasserstein distance and applying Gronwall's lemma that
		\begin{align}
		\mathcal{W}_{C_t, p}(\Phi^\nu(\mu^1)\vert_{[0,t]}, \Phi^\nu(\mu^2)\vert_{[0,t]})^p 
		\leq & \int_{\mathbb{R}^d \times C_T} \Vert S^{\mu^1}(x_0, \gamma) - S^{\mu^2}(x_0, \gamma)\Vert_{\infty:t}^p \;d\nu(x_0, \gamma) \nonumber\\
		\leq &  C_pK_b e^{tK_bC_p} \int_{0}^{t} \mathcal{W}_{C_s, p}(\mu^1\vert_{[0,s]}, \mu^2\vert_{[0,s]})^p ds. \label{eq:to iterate estimate}
		\end{align}
		Taking $t = T$, we have that
		\begin{equation}\label{estimate two of 1}
		\mathcal{W}_{C_T, p}(\Phi^\nu(\mu^1), \Phi^\nu(\mu^2))^p 
		\leq \tilde L \mathcal{W}_{C_T, p}(\mu^1, \mu^2)^p.
		\end{equation}
		With estimates \eqref{estimate one of 1} and \eqref{estimate two of 1} we have shown that $\Phi$ satisfies $1)$.
		
		To prove $2)$, we reiterate $k$ times the application $\Phi^\nu$ and we use \eqref{eq:to iterate estimate} to obtain
		\begin{align*}
		\mathcal{W}_{C_T, p}((\Phi^\nu)^k(\mu^1), (\Phi^\nu)^k(\mu^2))^p
		\leq & \tilde L^k \int_{0}^{T}\int_{0}^{t_k}\cdots \int_{0}^{t_2} \mathcal{W}_{C_{t_1},p}(\mu^1\vert_{[0,t_1]}, \mu^2\vert_{[0,t_1]})^p dt_1 \dots dt_k\\
		\leq & \tilde L^k \mathcal{W}_{C_{T},p}(\mu^1, \mu^2)^p \int_{0}^{T}\int_{0}^{t_k}\cdots \int_{0}^{t_2} dt_1 \dots dt_k\\
		\leq & \frac{(T\tilde L)^k}{k!} \mathcal{W}_{C_{T},p}(\mu^1, \mu^2)^p.
		\end{align*}
		By choosing $k>0$ large enough, we have that $c := \frac{(T\tilde L)^k}{k!} < 1$. This shows point $2)$ and concludes the proof.
	\end{proof}
	
	If the driving process is progressively measurable, then so is the solution:
	
	\begin{proposition}\label{prop:progr_meas}
		Let $(\mathcal{F}_t)_{t\geq0}$ be a right-continuous, complete filtration on $(\Omega,\mathcal{A},\mathbb{P})$ such that $\zeta$ is $\mathcal{F}_0$-measurable and $W$ is $(\mathcal{F}_{t})_{t\geq0}$-progressively measurable. Then the solution $X$ to \eqref{mckean vlasov equation} is also $(\mathcal{F}_t)_{t\geq0}$-progressively measurable.
	\end{proposition}
	
	\begin{proof}
		The proof is classical. Fix $t$ in $[0,T]$, then, $\mathbb{P}$-a.e.,~the restriction $X\vert_{[0,t]}=X\vert_{[0,t]}(\omega)$ on $[0,t]$ of the solution $X$ also solves \eqref{linear equation} on $[0,t]$ with inputs $\zeta$ and $W\vert_{[0,t]}$ (restriction of $W$ on $[0,t]$) and input measure $\mu\vert_{[0,t]}$ (pushforward of $\mu=\mathcal{L}(X)$ by the restriction on $[0,t]$). Therefore $X\vert_{[0,t]}(\omega)=S^{\mu\vert_{[0,t]}}_t(\zeta,W\vert_{[0,t]})$. Since $S^{\mu\vert_{[0,t]}}_t$ is $\mathcal{B}(\mathbb{R}^d)\otimes \mathcal{B}(C_t)$-measurable and $\zeta$ and $W\vert_{[0,t]}$ are $\mathcal{F}_t$-measurable, also $X\vert_{[0,t]}$ is $\mathcal{F}_t$-measurable, in particular $X\vert_{[0,t]}$ is $\mathcal{F}_t$-measurable. Hence $X$ is adapted and therefore progressively measurable by continuity of its paths.
	\end{proof}

	\subsection{Weak continuity}
	\label{section: weak continuity}
	
	In this note we are generally interested in proving quantitative convergence in the Wasserstein distance. However, one can show that the law of the solution of the mean field equation \eqref{mckean vlasov equation} is continuous in the weak topology of measures, with respect to the law of the inputs, in the spirit of \cite{MR780770}.
	
	\begin{assumption}
		\label{asm: weak metric}
		Given a Polish space $(E, d)$, we endow the space $\mathcal{P}(E)$ with a metric $\Pi_E$, with the following properties
			\begin{enumerate}[label=(\roman*), ref=\ref{asm: weak metric} (\roman*)]
			\item \label{asm: weak metric: compatible} 
			The metric $\Pi_E$ is complete and metrizes the weak convergence of measures.
			\item \label{asm: weak metric: Lipschitz}
			For any two random variables $X, X^\prime : \Omega \to E$, we have 
			\begin{equation*}
			\Pi_{E}(\mathcal{L}(X), \mathcal{L}(X^\prime)) \leq \mathbb{E}d(X, X^\prime).
			\end{equation*}
		\end{enumerate}
	\end{assumption}
	\begin{remark}
		Let $\operatorname{Lip}_1$ be the space of bounded and Lipschitz functions on $E$, as defined in Section \ref{section:notation}. Define the Kantorovich-Rubinstein metric as
		\begin{equation*}
		\Pi_{E}(\mu, \nu) := \sup_{\varphi \in \operatorname{Lip}_1} \int_{E}\varphi d(\mu - \nu),
		\end{equation*}
		This metric satisfies Assumption \ref{asm: weak metric}. Note that \ref{asm: weak metric: compatible} follows from \cite[Theorem 8.3.2 and Theorem 8.9.4]{bogachev2007measure}
	\end{remark}


	For the drift we assume the following.
	\begin{assumption}
	\label{assumption weak}
	The drift $b: [0,T] \times \mathbb{R}^d \times \mathcal{P}(\mathbb{R}^d) \to \mathbb{R}^d$ is a measurable function and there exists a constant $K$ such that, 
		\begin{itemize}
			\item (Lipschitz continuity)
			\begin{equation*}
			\vert b(t, x, \mu) - b(t, x^\prime, \mu^\prime) \vert
			\leq K \left(\vert x - x^\prime \vert
			+ \Pi_{\mathbb{R}^d}(\mu, \mu^\prime)\right),
			\end{equation*}
			$\forall t \in [0,T], x, x^\prime \in \mathbb{R}^d, \mu, \mu^\prime \in \mathcal{P}(\mathbb{R}^d)$.
			\item (boundedness)
			\begin{equation*}
						\vert b(t, x, \mu)\vert
						\leq K,
			\end{equation*}
						$\forall t \in [0,T], x \in \mathbb{R}^d, \mu \in \mathcal{P}(\mathbb{R}^d)$.
		\end{itemize}
	\end{assumption}
	\begin{remark}
		\label{rem: linear drift}
Assume that there exists a function $B: \mathbb{R}^d \times \mathbb{R}^d \to \mathbb{R}^d$ such that there exists a constant $C>0$,
\begin{equation*}
\vert B(x, y) \vert 
\leq C,
\quad\vert B(x, y) - B(x^\prime, y^\prime) \vert 
\leq C\left(
\vert x - x^\prime \vert + \vert y - y^\prime \vert
\right),
\quad \forall x,x^\prime, y, y^\prime \in \mathbb{R}^d,
\end{equation*}
and the drift satisfies $b(t,x, \mu) := \int_{\mathbb{R}^d} B(x, y) \mu(dy)$. Then $b$ satisfies Assumptions \ref{assumption weak}, with $K=3C$. This is the case treated in \cite{MR780770}.
	\end{remark}
\begin{lemma}
	\label{weak cont:existence}
	Given $\nu \in \mathcal{P}(\mathbb{R}^d \times C_T)$, the solution map
	\begin{equation}
	\begin{array}{cccc}
	S^\nu: & \mathbb{R}^d\times C_T & \rightarrow  & C_T\\
	& (x_0, \gamma) & \mapsto & S^\nu(x_0, \gamma),
	\end{array}
	\end{equation}
	to the ODE
	\begin{equation}
	\label{weak cont:ode}
	x_t = x_0 + \int_0^t b(s, x_s, (x_s)_{\#}\nu) ds + \gamma_t.
	\end{equation} 
	is well defined.
\end{lemma}
\begin{proof}
	We prove the lemma by iteration. For a fixed $x_0, \gamma \in \mathbb{R}^d \times C_T$, define $x^0_t := x_0 + \gamma_t$, and $x^{n+1}_t$ defined implicitly as $x^{n+1}_t = x_0 + \int_0^t b(s, x^{n+1}_s, (x^n_s)_{\#}\nu) ds + \gamma_t$. Clearly, for every $n\in \mathbb{N}$, the function $(x_0, \gamma) \mapsto x^n$ is well defined and measurable.
	
	We compute the following, for $t\in [0,T]$, using Assumption \ref{assumption weak}, Gronwall's Lemma and Assumption \ref{asm: weak metric: Lipschitz}
	\begin{equation*}
	\vert x^n_t - x^{n+1}_t \vert 
	\leq K e^{Kt}\int_{0}^{t}\Pi_{\mathbb{R}^d}((x^{n-1}_s)_{\#} \nu, (x^{n}_s)_{\#} \nu) ds
	\leq K e^{Kt}\int_{0}^{t} \int_{\mathbb{R}^d \times C_T} \vert x^{n-1}_s - x^{n}_s\vert d\nu ds.
	\end{equation*}
	Iterating this inequality down to $n=0$, we obtain that there exists a positive constant $C(T, K)$, independent of $n$, such that
			\begin{align*}
			\vert x^n_t - x^{n+1}_t \vert 
			\leq  \frac{C(T, K)^n}{n!}.
			\end{align*}
Hence, we have that, for every $x_0, \gamma \in \mathbb{R}^d \times C_T$, the sequence $(x^n(x_0, \gamma))_{n\geq 0}$ is Cauchy in $(C_T, \Vert \cdot \Vert_{\infty})$. Indeed, for $\epsilon >0$, there exists $m>0$ big enough, such that for every $n \geq m$, 
\begin{equation*}
\Vert x^m - x^{n} \Vert_{\infty} 
\leq \sum_{i=m}^{n-1} \Vert x^i - x^{i+1} \Vert_{\infty}
\leq  \sum_{i=m}^{\infty}  \frac{C(T, K)^i}{i!} < \epsilon.
\end{equation*}

We call $x(x_0, \gamma) \in C_T$ its limit as $n\to \infty$. The pointwise limit of Borel measurable functions is measurable, hence $(x_0, \gamma) \mapsto x$ is also measurable and $(x_s)_{\#}\nu$ is well-defined. We can thus pass to the limit in equation \eqref{weak cont:ode} to show that $x$ is a solution to it.

To prove uniqueness, let $x$ and $y$ be two solutions with the same inputs $x_0, \gamma$ and $\nu$. As before, we can compute, for $t \in [0,T]$,
\begin{equation*}
\vert x_t - y_t \vert 
\leq K e^{Kt}\int_{0}^{t} \int_{\mathbb{R}^d \times C_T} \vert x_s - y_s\vert d\nu ds.
\end{equation*}
Integrating in $d\nu(\gamma)$ and applying Gronwall's lemma we get that the righ hand side vanishes. Hence, $x$ and $y$ are the same for all $t \in [0,T]$ and $\gamma \in C_T$.
\end{proof}


\begin{lemma}
	The function
		\begin{equation}
		\begin{array}{cccc}
		\Psi: & (\mathcal{P}(\mathbb{R}^d\times C_T), \Pi_{{\mathbb{R}^d}\times C_T}) & \rightarrow  & (\mathcal{P}(C_T), \Pi_{C_T})\\
		& \nu & \mapsto & (S^\nu)_{\#}\nu,
		\end{array}
		\end{equation}
		is continuous.
		 By Assumption \ref{asm: weak metric: compatible}, this is equivalent to continuity with respect to the topology induced by the weak convergence of measures.
\end{lemma}
	
\begin{proof}
	Let $(\nu^n)_{n\geq 0} \subset \mathcal{P}(\mathbb{R}^d \times C_T)$ be a sequence of probability measures that converges weakly to $\nu \in \mathcal{P}(\mathbb{R}^d \times C_T)$. From Skohorokhod representation theorem, there exists a probability space $(\Omega, \mathcal{A}, \mathbb{P})$ and a sequence $(\zeta^n, W^n): \Omega \to \mathbb{R}^d \times C_T$ of random variables distributed as $\nu^n$ that converges almost surely to a random variable $(\zeta, W)$ distributed as $\nu$. 
	
	Let $X^n := S^{\nu^n}(\zeta^n, W^n)$. By definition, $\mu^n := \mathcal{L}(X^n) = \Psi(\nu^n)$ and $X^n$ solves the following SDE in the sense of Definition \ref{def solution},
	\begin{equation*}
	X^n_t = \zeta^n + \int_{0}^{t} b(s, X^n_s, \mathcal{L}(X^n_s))ds + W^n_t.
	\end{equation*}
	It is easy to check that the random variables $X^n$ are equicontinuous and equibounded and deduce that the family $\mu^n$ is tight in $C_T$. With an abuse of notation, assume that $(\mu^n)_{n\geq0}$ is a subsequence that converges weakly to some $\mu \in \mathcal{P}(C_T)$, and $(X^n)_{n\geq 0}$ such that $\mathcal{L}(X^n) = \mu^n$. By using the equation, one can check that $(X^n(\omega))_{n\geq 0}$ is a Cauchy sequence in $C_T$ for $\mathbb{P}-a.e.$ $\omega$. Let $X$ be the almost sure limit of $X^n$, as $n\to \infty$. Clearly, $\mu^n$ converges weakly to $\mathcal{L}(X)$, hence $\mathcal{L}(X) = \mu$. Passing to the limit in the equation, we can see that $\mu = \mathcal{L}(X) = \Psi(\nu)$. This concludes the proof.
\end{proof}

\subsection{C\`adl\`ag drivers}
In this section we follow the same reasoning as Section \ref{section: weak continuity} to study the case when the drivers are discontinuous processes in $(D_T, \sigma)$. We first set some notation and recall some results about c\`adl\`ag functions.

Given $t \in [0, T]$, the projection $\pi_t$ is defined, analogously to the continuous case, as the function $\pi_t: D_T \to \mathbb{R}^d$ as $\pi_t(\gamma) := \gamma(t)$.

	\begin{definition}
		\label{def: cadlag modulus}
	For a function $\gamma \in D_T$, we define its c\`adl\`ag modulus as a function of $\delta \in (0,1)$,
	\begin{equation*}
	w_\gamma(\delta) =  \inf_\Pi \max_{1\leq i \leq n} \sup_{t_{i-1} \leq s \leq t < t_i} \vert \gamma_{s,t} \vert,
	\end{equation*}
	where the infimum is taken over all the partitions $\Pi$ with mash size bigger than $\delta$.
\end{definition}
Then we have the following lemma, from \cite[equation (13.3)]{billingsly1999convergence}.
\begin{lemma}
	\label{lem:convergence of cadlag marginals}
	Let $(\nu^n)_{n\geq 0} \subset \mathcal{P}(D_T)$ be a sequence of probability measures converging weakly to $\nu \in \mathcal{P}(D_T)$, then there exists a set $T_{\nu} \subset [0,T]$ of full Lebesgue measure (actually $T_{\nu}^c$ is at most countable) such that $\nu^n_t$ converges weakly to $\nu_t$, for all $t\in T_{\nu}$.
\end{lemma}

Given a Polish space $(E,d)$, we use once again the notation $\Pi_{E}$ to denote a distance on $\mathcal{P}(E)$ that satisfies Assumption \ref{asm: weak metric}.

For the drift we assume the following.
\begin{assumption}
	\label{assumption: weak jump}
	The drift $b: \mathbb{R}^d \times \mathcal{P}(\mathbb{R}^d) \to \mathbb{R}^d$ is a measurable function and there exists a constant $K$ such that, 
	\begin{itemize}
		\item (Lipschitz continuity)
		\begin{equation*}
		\vert b(x, \mu) - b(x^\prime, \mu^\prime) \vert
		\leq K \left(\vert x - x^\prime \vert
		+ \Pi_{\mathbb{R}^d}(\mu, \mu^\prime)\right),
		\end{equation*}
		$\forall x, x^\prime \in \mathbb{R}^d, \mu, \mu^\prime \in \mathcal{P}(\mathbb{R}^d)$.
		\item (boundedness)
		\begin{equation*}
		\vert b(x, \mu)\vert
		\leq K,
		\end{equation*}
		$\forall x \in \mathbb{R}^d, \mu \in \mathcal{P}(\mathbb{R}^d)$.
	\end{itemize}
\end{assumption}
\begin{remark}
	The function $b$ defined in Remark \ref{rem: linear drift} also satisfies Assumptions \ref{assumption: weak jump}.
\end{remark}

\subsubsection{Well-posedness and continuity}
We have the following results, analogously to Section \ref{section: weak continuity}.
\begin{lemma}
	\label{lem: jump: existence}
	Let $b$ satisfy Assumptions \ref{assumption: weak jump}. Given $\nu \in \mathcal{P}(\mathbb{R}^d \times D_T)$, the solution map
	\begin{equation}
	\begin{array}{cccc}
	S^\nu: & \mathbb{R}^d\times D_T & \rightarrow  & D_T\\
	& (x_0, \gamma) & \mapsto & S^\nu(x_0, \gamma),
	\end{array}
	\end{equation}
	to the ODE
	\begin{equation}
	\label{jump: ode}
	x_t = x_0 + \int_0^t b(x_s, (x_s)_{\#}\nu) ds + \gamma_t.
	\end{equation} 
	is well defined.
\end{lemma}
\begin{proof}
	The proof of this lemma follows exactly the proof of Lemma \ref{weak cont:existence}. We define the sequence $(x^n)_{n\in\N} \subset D_T$ and show that it is a Cauchy sequence in the uniform norm $\|\cdot \|_{\infty}$. By taking $\lambda(t) = t$ in the definition of $\sigma$, equation \eqref{def: sko metric}, one notices immediately that $\sigma$ is bounded by the distance induced by $\| \cdot \|_{\infty}$. Hence, $(x^n)_{n\in\N}$ is a Cauchy sequence in $\sigma$ and the conclusion follows as in Lemma \ref{weak cont:existence}.
\end{proof}

\begin{lemma}
	\label{lem: jump continuity}
	The function
	\begin{equation}
	\begin{array}{cccc}
	\Psi: & (\mathcal{P}(\mathbb{R}^d\times D_T), \Pi_{{\mathbb{R}^d}\times D_T}) 
	& \rightarrow  & (\mathcal{P}(D_T), \Pi_{D_T})\\
	& \nu & \mapsto & (S^\nu)_{\#}\nu,
	\end{array}
	\end{equation}
	is continuous. By assumption, this is equivalent to continuity with respect to the topology induced by the weak convergence of measures.
\end{lemma}

\begin{proof}
	Let $(\nu^n)_{n\geq 0} \subset \mathcal{P}(\mathbb{R}^d \times D_T)$ be a sequence of probability measure that converges weakly to $\nu \in \mathcal{P}(\mathbb{R}^d \times D_T)$. From Skohorokhod representation theorem, there exists a probability space $(\Omega, \mathcal{A}, \mathbb{P})$ and a sequence $(\zeta^n, W^n): \Omega \to \mathbb{R}^d \times D_T$, for $n\geq 0$, of random variables distributed as $\nu^n$ which converges almost surely to a random variable $(\zeta, W)$ distributed as $\nu$. 
	
	Let $X^n := S^{\nu^n}(\zeta^n, W^n)$. By definition, $\mu^n := \mathcal{L}(X^n) = \Psi(\nu^n)$ and $X^n$ solves the following SDE pathwise,
	\begin{equation*}
	X^n_t = \zeta^n + \int_{0}^{t} b(X^n_s, \mathcal{L}(X^n_s))ds + W^n_t.
	\end{equation*}
	By construction, the laws of $W^n$ are tight. Equivalently, by \cite[Theorem 13.2]{billingsly1999convergence}, they satisfy
	\begin{equation}
	\label{eq: equibounded}
	\lim_{a\to\infty} \limsup_{n\to\infty} \mathbb{P} \{ \|W^n\|_{\infty} \geq a\} = 0,
	\end{equation}
	\begin{equation}
	\label{eq: equijumps}
	\forall \epsilon > 0,
	\qquad
	\lim_{\delta \to 0}\limsup_{n\in\N}\mathbb{P} \{ w_{W^n}(\delta) \geq \epsilon \} = 0.
	\end{equation}
	It follows from Assumption \ref{assumption: weak jump} that the random variables $X^n$ also satisfy \eqref{eq: equibounded} and \eqref{eq: equijumps}. Thus, we deduce that the family $\mu^n$ is tight in $\mathcal{P}(\mathbb{R}^d \times D_T)$. 
	
	With an abuse of notation, assume that $(\mu^n)_{n\geq0}$ is a subsequence that converges weakly to some $\mu \in \mathcal{P}(D_T)$, and $(X^n)_{n\geq 0}$ such that $\mathcal{L}(X^n) = \mu^n$. By using the equation, we now check that $(X^n(\omega))_{n\geq 0}$ is a Cauchy sequence in $(D_T, \sigma)$ for $\mathbb{P}-a.e.$ $\omega$. 
	
	First observe that Lemma \ref{lem:convergence of cadlag marginals} and Lebesgue dominated convergence imply that $\int_{0}^{T}\Pi_{\mathbb{R}^d}(\mu^n_s, \mu_s) ds \to 0$, as $n\to \infty$. Hence $(\mu^n)_{n\in\N} \subset L^1([0,T], \mathcal{P}(\mathbb{R}^d))$ is a Cauchy sequence. Let now $\Omega^0 \subset \Omega$ be a set of full measure such that $(\zeta^n(\omega), W^n(\omega)) \to (\zeta(\omega), W(\omega))$, for all $\omega \in \Omega^0$, as $n\to \infty$.
	
	Fix $\omega \in \Omega^0$, $\epsilon >0$ there exists $N >0$, such that for all $m,n \geq N$, we have
	\begin{equation*}
	\sigma(W^n(\omega), W^m(\omega)) < \epsilon,
	\quad
	\vert \zeta^n(\omega) - \zeta^m(\omega) \vert < \epsilon,
	\quad
	\int_{0}^{T}\Pi_{\mathbb{R}^d}(\mu^n_s, \mu^m_s) ds < \epsilon.
	\end{equation*}
	Moreover, since the sequence $(W^n(\omega))$ converges, it is pre-compact in $D_T$. It follows from \cite[Theorem 12.3]{billingsly1999convergence} that
	\begin{equation}
	\label{eq: cadlag equcontinuity}
	\lim_{\delta \to 0} \sup_{n} w(W^n(\omega), \delta) = 0.
	\end{equation}
	It follows from Assumption \ref{assumption: weak jump} that $| X^n_{s,t}(\omega) | \leq | W^n_{s,t}(\omega) | + K |t - s|$, for all $t,s \in [0,T]$ and $n\in \N$. Hence, one can replace $W^n$ with $X^n$ in \eqref{eq: cadlag equcontinuity}.
	We omit now the dependence of the random variables from $\omega$.
	There exists $\bar \delta > 0$, such that for every $0<\delta<\bar\delta$, $\sup_{n}w(X^n, \delta) < \epsilon$. 
	We can choose $\delta = \epsilon \wedge \bar \delta$, and $\lambda := \lambda(\omega, \delta, m, n)$ such that
	\begin{equation}
	\label{eq:control on w}
	\Vert \lambda \Vert + \Vert W^n - W^m \circ \lambda \Vert_{\infty} < \delta < \epsilon.
	\end{equation}
	It follows from \cite[equation (12.17)]{billingsly1999convergence} that, for all $t\in [0,T]$,
	\begin{equation*}
	|\lambda_t - t| \leq e^{\| \lambda \|} - 1 \leq e^\delta - 1 \approx \delta,
	\end{equation*}
	Hence, for any $t\in [0,T]$ and for any partition $\Pi$ of $[0,T]$ of mesh size bigger than $\delta$ we find at most one point of the partition between $t$ and $\lambda_t$, which gives
	\begin{equation*}
	\Vert X^m - X^m \circ \lambda \Vert_{\infty} < 2w(X^m, \delta) < 2\sup_{n}w(X^n, \delta) < 2\epsilon.
	\end{equation*}
	We note that, for all $t\in [0,T]$,
	\begin{equation*}
	\left \vert\int_{\lambda_t \vee t}^{\lambda_t \wedge t} ds \right \vert \leq |\lambda_t - t| \leq e^{\| \lambda \|} - 1,
	\end{equation*}
	where the last inequality follows from \cite[equation (12.17)]{billingsly1999convergence}.
	We can thus compute the following, for $t \in [0,T]$,
	\begin{align*}
	\left \vert \int_{0}^t b(X_s^n, \mu^n_s)ds - \int_{0}^{\lambda_t} b(X_s^m, \mu^m_s)ds \right \vert
	\leq & K \left \vert\int_{\lambda_t \vee t}^{\lambda_t \wedge t} ds \right \vert 
	+ \int_{0}^{\lambda_t \wedge t} \left \vert b(X_s^n, \mu^n_s) -  b(X_s^m, \mu^m_s) \right \vert ds \\
	\leq & K(e^{\| \lambda \|} - 1)
	+ K\int_0^T \Pi_{{\mathbb{R}^d}} (\mu^n_s, \mu^m_s) ds
	+ K\int_0^T \vert X^n_s - X^m_s \vert ds\\
	\leq & K(e^{\| \lambda \|} - 1)
	+ K\epsilon
	+ K\int_0^T \vert X^n_s - (X^m \circ \lambda)_s \vert ds \\
	& + K\int_0^T \vert X^m_s - (X^m \circ \lambda)_s \vert ds \\
	\lesssim &
	4\epsilon
	+ \int_0^T \vert X^n_s - (X^m \circ \lambda)_s \vert ds 
	\end{align*}
	From which we deduce
	\begin{align*}
	\vert X_t^n - (X^m \circ \lambda)_t \vert
	\leq & \vert \zeta^n - \zeta^m \vert
	+ \vert W^n_t - (W^m \circ \lambda)_t \vert
	+ \left \vert \int_{0}^t b(X_s^n, \mu^n_s)ds - \int_{0}^{\lambda_t} b(X_s^m, \mu^m_s)ds \right \vert \\
	\lesssim & 
	\; \epsilon 
	+ \Vert W^n - W^m\circ \lambda \Vert
	+ \int_0^T \vert X^n_s - (X^m \circ \lambda)_s \vert ds 
	\end{align*}
	We add $\Vert \lambda \Vert$ on both sides, apply Gronwall's Lemma and inequality \eqref{eq:control on w} to obtain
	\begin{equation*}
	\sigma(X^n(\omega), X^m(\omega)) < C(T, K) \epsilon.
	\end{equation*}
	Hence, we have that $X^n(\omega)$ is a Cauchy sequence in $(D_T, \sigma)$, for $\omega \in \Omega^0$.
	
	Let $X$ be the almost sure limit of $X^n$, as $n\to \infty$. The laws $\mu^n$ converge weakly to $\mathcal{L}(X)$, hence $\mathcal{L}(X) = \mu$. Passing to the limit in the equation, we can see that $\mu = \mathcal{L}(X) = \Psi(\nu)$. This concludes the proof.
\end{proof}

	\section{Applications}
	\label{section:applications}
	
	\subsection{Particle approximation}
	\label{section:particle approximation}
	
	In this section we show how the results in Section \ref{section:main result} yield a convergence result for a particle system associated with the McKean-Vlasov equation.
	
	Given inputs $\bar{\zeta}$ and $\bar{W}$ (on a probability space $(\Omega, \mathcal{A}, \mathbb{P})$), we consider the following McKean-Vlasov equation
	\begin{equation}\label{limit mckean vlasov}
	\left\{
	\begin{array}{l}
	d\bar{X}_t = b(t, \bar{X}_t, \mathcal{L}(\bar{X_t})) dt + d\bar{W}_t\\
	X_0 = \bar{\zeta}.
	\end{array}
	\right.	
	\end{equation}	
	To this, given $N\in \mathbb{N}$, we associate the corresponding interacting particle system (on a probability space $(\Omega, \mathcal{A}, \mathbb{P})$), 
	\begin{equation}\label{classical system}
	\left\{
	\begin{array}{l}
	dX^{i,N}_t = b(t, X^{i,N}_t, \frac1N \sum_{i=1}^{N}\delta_{X_t^{i,N}}) dt + dW^{i,N}_t,\\
	X^{i,N}_0 = \zeta^{i,N},
	\end{array}
	\right.	
	\quad i=1, \dots, N
	\end{equation}
	with given input 
	\begin{equation*}
	\begin{array}{cccc}
	(\zeta^{(N)}, W^{(N)}) : & \Omega & \to & (\mathbb{R}^d \times C_T)^N\\
	&\omega &\mapsto & (\zeta^{i,N}(\omega), W^{i,N}(\omega))_{1\leq i \leq N}.
	\end{array}
	\end{equation*}

	
	For a given $N\in \mathbb{N}$ and an $N$-dimensional vector $Y^{(N)} = (Y^1, \cdots, Y^N)$ with entries in a Polish space $E$, we define the empirical measure associated with $Y^{(N)}$ as
	\begin{equation*}
	L^N(Y^{(N)}) := \frac1N \sum_{i=1}^{N}\delta_{Y^i}.
	\end{equation*}
	As pointed out in the introduction, the main argument of Cass-Lyons/Tanaka approach is that the particle system \eqref{intro:interacting particles} can be interpreted as the limiting McKean-Vlasov equation \eqref{intro:mckean vlasov equation} by using a transformation of the probability space and the input data.
	The main result Theorem \ref{main theorem} not only implies well-posedness of both McKean-Vlasov and particle approximation, but also allows to deduce convergence of the particle system from convergence of the corresponding signals, something which is usually easy to verify, for example, if the signals are empirical measures of independent noises.
	
	Now we show how to interpret equations \eqref{limit mckean vlasov} and \eqref{classical system} as generalized McKean-Vlasov equation \eqref{mckean vlasov equation}. Clearly \eqref{limit mckean vlasov} is \eqref{mckean vlasov equation} with inputs $\bar{\zeta}$ and $\bar{W}$. For \eqref{classical system}, for fixed $N\in\mathbb{N}$, we consider the space $(\Omega_N, \mathcal{A}_N, \mathbb{P}_N)$, where $\Omega_N := \{1, \dots, N\}$, $\mathcal{A}_N:=2^{\Omega_N}$ and $\mathbb{P}_N := \frac{1}{N}\sum_{i=1}^{N}\delta_i$. On this space, we can identify any $N$-uple $Y^{(N)}=(Y^1, \dots, Y^N) \in E^N$, as a random variable $\Omega_N \ni i\mapsto Y^i \in E$. With this identification, the law of $Y^{(N)}$ on $\Omega_0$ is precisely the empirical measure associated with $Y^{(N)}$, namely $L^N(Y^{(N)})$.
	Indeed, for each continuous and bounded function $\varphi$ on $E$, we have
	\begin{align*}
	\mathbb{E}_{\mathbb{P}_N}[\varphi(Y^{(N)})]
	= \sum_{i=1}^{N}\frac1N  \varphi(Y^i)  
	= L^N(Y^{(N)})(\varphi).
	\end{align*}
	We assume that $(\zeta^{(N)}(\omega),W^{(N)}(\omega))$ is valued in $(\mathbb{R}^d\times C_T)^N$ for every $N$ and for every $\omega\in \Omega$. We fix $\omega\in \Omega$ and $N$ and we apply the previous argument to the $N$-uples 
	\begin{align*}
	(\zeta^{(N)}, W^{(N)})(\omega) & =((\zeta^{1,N}, W^{1,N})(\omega), \dots, (\zeta^{N,N}, W^{N,N})(\omega)), \\
	X^{(N)}(\omega) & = (X^{1,N}(\omega),\ldots, X^{N,N}(\omega)).
	\end{align*}
	 For fixed $\omega \in \Omega$, the law of $(\zeta^{(N)}(\omega),W^{(N)}(\omega))$ on $\Omega_N$ is the empirical measure $L^N(\zeta^{(N)},W^{(N)})(\omega)$ and the law of $X^{(N)}(\omega)$ on $\Omega_N$ is the empirical measure $L^N(X^{(N)})(\omega)$, which appears exactly in \eqref{classical system}, projected at time $t$. Hence, for fixed $\omega$ in $\Omega$, the interacting particle system \eqref{classical system} is the generalized McKean-Vlasov equation \eqref{mckean vlasov equation}, defined on the space $(\Omega_N, \mathcal{A}_N, \mathbb{P}_N)$ and driven by the empirical measure $L^N(\zeta^{(N)},W^{(N)})(\omega)$.

	We are ready to apply Theorem \ref{main theorem} to obtain the following result, which ties the convergence of the particles to the convergence of the inputs. An immediate consequence is that the empirical measure of the particle system converges if the input converges: no independence or exchangeability are required.
	
	\begin{theorem}\label{classical mean field result}
		Let $p \in [1,\infty)$ and assume \ref{assumption}. Let $(\Omega, \mathcal{A}, \mathbb{P})$ be a probability space. For a fixed $N\in\mathbb{N}$, let $(\zeta^{(N)}, W^{(N)}) = (\zeta^{i,N}, W^{i,N})_{1\leq i \leq N}:\Omega \to (\mathbb{R}^d \times C_T)^N$ be a family of random variables. Let $\bar \zeta \in L^p(\Omega, \mathbb{R}^d)$ and $\bar W \in L^p(\Omega, C_T)$.
		Then,
		\begin{enumerate}[label=\roman{*} ]
			\item\label{item 1 classical mean field} for every $\omega \in \Omega$, there exists a unique pathwise solution $X^{(N)}(\omega)$ in the sense of Definition \ref{def solution} to the interacting particle system \eqref{classical system}. Moreover, $\omega \mapsto X^{(N)}(\omega)$ is $\mathcal{A}$-measurable.
			\item\label{item 2 classical mean field} there exists a unique pathwise solution $\bar X$ in the sense of Definition \ref{def solution} to equation \eqref{limit mckean vlasov}. 
			\item\label{item 3 classical mean field} there exists a constant $C$ depending on $b$ such that 	 for all $N \geq 1$, for $\mathbb{P}$-a.e.~$\omega \in \Omega$,
			\begin{equation}\label{mean field control}
			\mathcal{W}_{C_T, p}(L^N(X^{(N)}(\omega)), \mathcal{L}(\bar{X}))^p \leq C \mathcal{W}_{\mathbb{R}^d \times C_T, p}(L^N(\zeta^{(N)}(\omega), W^{(N)}(\omega)), \mathcal{L}(\bar{\zeta}, \bar{W}))^p.
			\end{equation}
		\end{enumerate}
	\end{theorem}
	
	\begin{proof}
		Let $N\in \mathbb{N}$. Fix $\omega \in \Omega$, we apply Theorem \ref{main theorem} in the following setting
		\begin{equation*}
		(\Omega^1, \mathcal{A}^1, \mathbb{P}^1) := (\Omega_N, \mathcal{A}_N, \mathbb{P}_N),
		\quad (\zeta^1, W^1)(\omega) := (\zeta^{(N)}(\omega), W^{(N)}(\omega)),
		\end{equation*}
		\begin{equation*}
		(\Omega^2, \mathcal{A}^2, \mathbb{P}^2) := (\Omega, \mathcal{A}, \mathbb{P}),
		\quad (\zeta^2, W^2) := (\bar \zeta, \bar W).
		\end{equation*}
		The finite $p$-moment condition is satisfied by $(\bar \zeta, \bar W)$ by assumption and also by $(\zeta^{(N)}(\omega), W^{(N)}(\omega))$, since
		\begin{align*}
		\Vert (\zeta^1, W^1)(\omega)  \Vert_{L^p(\Omega^1)}^p
		= & \mathbb{E}_{\mathbb{P}_N}\left[ \vert\zeta^{(N)}(\omega)\vert^p + \Vert W^{(N)}(\omega)\Vert_{\infty}^p \right]\\
		= & \frac1N \sum_{i=1}^{N}\vert\zeta^i(\omega)\vert^p + \frac1N \sum_{i=1}^{N}\Vert W^i(\omega)\Vert_{\infty}^p
		< +\infty.
		\end{align*}
		Since the assumptions on the drift $b$ are also satisfied, Theorem \ref{main theorem} establishes the existence of solutions $X^1 (\omega) =: X^{(N)}(\omega)$ and $X^2 =: \bar X$. Moreover the map $\Psi$ is continuous, hence $\omega \mapsto L^{(N)}(X^{(N)})(\omega)$ is $\mathcal{A}$-measurable, which makes $X^{(N)}(\omega) := S^{L^{(N)}(X^{(N)})(\omega)}(\zeta(\omega), W^{(N)}(\omega))$ measurable. This gives \eqref{item 1 classical mean field} and \eqref{item 2 classical mean field}. Theorem \eqref{main theorem} also gives exactly the inequality in \eqref{item 3 classical mean field}. The proof is complete.
	\end{proof}
	
	\begin{remark}
		We stress out that, when looking at the particle system, we are applying Theorem \ref{main theorem} on the discrete space, for a fixed $\omega$, and the law that appears on the drift is the empirical measure at fixed $\omega$.
	\end{remark}
	
	\begin{remark}
		In the proof of point \ref{item 3 classical mean field} of Theorem \ref{classical mean field result}, we can actually get the bound for \textit{every} $\omega$ if we use the pathwise solution $X^{(N)}(\omega)$ (in the sense of Definition \ref{def solution}), as this satisfies \eqref{classical system} for \textit{every} $\omega$. However, the ``$\mathbb{P}$-a.s.'' is required when dealing with a solution to the interacting particle system \eqref{classical system} in the usual probabilistic sense, where \eqref{classical system} is required to hold only $\mathbb{P}$-a.s..
	\end{remark}
	
	\subsection{Classical mean field limit}
	\label{section:classical mean field limit}
	
	Now we specialize the previous result in the case of i.i.d.~inputs, recovering the classical result by Sznitman \cite{sznitman1991topics}:	
	\begin{corollary}
		\label{corollary:classical mean field}
		Given a filtered probability space $(\Omega, \mathcal{A}, (\mathcal{F}_t)_{t\geq 0}, \mathbb{P})$ (with the standard assumptions) and $p \in (1, \infty)$ let $(\zeta^{i})_{i\geq 1}\subset L^{p}(\Omega, \mathbb{R}^d)$, be a family of i.i.d.~random variables which are $\mathcal{F}_0$-measurable and $ (W^i)_{i\geq 1}$ be a family of independent adapted Brownian motions. Moreover, let $(\bar \zeta, \bar W) \in L^p(\Omega, \mathbb{R}^d \times C_T)$ be an independent copy of $(\zeta^1, W^1)$.
		Then the solutions $X^{(N)}$ and $\bar X$ to the interacting particles system \eqref{classical system} and the McKean-Vlasov SDE \eqref{limit mckean vlasov}, respectively, given by Theorem \ref{classical mean field result}, are progressively measurable and we have the following convergence
		\begin{equation}\label{mean field weak convergence} 
		L^N(X^{(N)}) \overset{\ast}{\rightharpoonup} \mathcal{L}(\bar{X}), \quad  \mathbb{P}-a.s.
		\end{equation}
	\end{corollary}
	
	\begin{remark}
		The classical case when $b$ is a convolution with a regular kernel, say $b(t, x, \mu) = (K \ast \mu)(x)$, is treated here, as $b$ in this case satisfies the assumption of Theorem \ref{classical mean field result}.
	\end{remark}
	
	\begin{proof}[Proof of Corollary \ref{corollary:classical mean field}]
		Progressive measurability for the particle system \eqref{classical system} follows from \eqref{item 2 of main} of Theorem \ref{main theorem} and is a consequence of Proposition \ref{prop:progr_meas} for the McKean-Vlasov SDE \eqref{limit mckean vlasov}.
		
		We prove now the convergence. First recall that Theorem \ref{classical mean field result}, and in particular inequality \eqref{mean field control}, applies in this case. Hence, if we can prove that the right-hand-side of \eqref{mean field control} goes to zero, we have the desired convergence \eqref{mean field weak convergence}.
		
		Hence, by Lemma \ref{appendix:iid into wasserstein}, we deduce the convergence in $p^\prime$-Wasserstein, for every $p^\prime \in (1, p)$. This is the convergence of the right-hand-side of \eqref{mean field control}. The proof is complete.
		%
	\end{proof}
	
	\subsection{Mean field with common noise}
	\label{section:common noise}
	
	In this section we study a system of interacting particles with common noise. We consider the following system on the space $(\bar \Omega, \bar{\mathcal{A}}, \bar{\mathbb{P}})$,
	\begin{equation}\label{common noise particles}
	\left\{
	\begin{array}{l}
	dX^{i,N}_t = b(t, X^{i,N}_t, \frac1N \sum_{i=1}^{N}\delta_{X_t^{i,N}}) dt + dW^{i}_t + dB_t\\
	X^{i,N}_0 = \zeta^{i}.
	\end{array}
	\right.	
	i=1, \dots, N
	\end{equation}
	Here $(\zeta^{i})_{i=1,\dots, N} \subset L^p(\bar\Omega, \mathbb{R}^d)$ is a family of i.i.d.~random variables. This system represents $N$ interacting particles where each particle is subject to the interaction with the others as well as some randomness. There are two sources of randomness, one which acts independently on each particle and is represented by the independent family of identically distributed random variables $W^{(N)} = (W^{i})_{1\leq i\leq N} \subset L^p(\bar\Omega, C_T)$. The second source of randomness is the same for each particle and is represented by the random variable $B\in L^p(\bar\Omega, C_T)$, which is assumed to be independent from the $W^{i}$. Usually $W^{i}$ and $B$ are Brownian motions, but it is not necessary to assume it here. The Brownian motion case was considered in \cite{CogFla2016}.
	
	Our aim is to prove that the empirical measure associate to the system converges, as $N\to \infty$, to the conditional law, given $B$, of the solution of the following McKean-Vlasov SDE
	\begin{equation}\label{mckean vlasov common noise}
	\left\{
	\begin{array}{l}
	d\bar{X}_t = b(t, \bar{X}_t, \mathcal{L}(\bar{X_t} \vert B)) dt + d\bar{W}_t + dB_t\\
	\bar X_0 = \bar{\zeta}.
	\end{array}
	\right.	
	\end{equation}
	Here $\bar \zeta$ is a random variable on $\mathbb{R}^d$ and $\bar{W}$ is random variables on $C_T$ distributed as $\zeta^1$ and $W^1$ respectively.
	We denote by $\mathcal{L}(X \vert B)$ the conditional law of $X$ given $B$. Our result is the following.
	\begin{corollary}
		Let $p\in [1,\infty)$, $p^\prime \in (p, \infty)$, and assume \ref{assumption}. Let $(\bar\Omega, \bar{\mathcal{A}}, \bar{\mathbb{P}})$ be a probability space. On this space we consider independent families $\zeta^{(N)} = (\zeta^{i})_{1\leq i\leq N}\subset L^{p^\prime}(\bar \Omega, \mathbb{R}^d)$, $W^{(N)} = (W^{i})_{1\leq i\leq N} \in L^{p^\prime}(\bar\Omega, C_T)$ of i.i.d.~random variables. Let $\bar \zeta$ be distributed as $\zeta^{i,N}$ and let $\bar W$ be distributed as $W^{i,N}$ and independent of $\bar \zeta$. Moreover, assume that $B\in L^p(\bar\Omega, C_T)$ is a random variable independent from the others. Then there exists a solution $X^{(N)} \in L^p(\bar \Omega, (C_T)^N)$ to equation \eqref{common noise particles} and a solution $\bar X\in L^p(\bar\Omega, C_T)$ to equation \eqref{mckean vlasov common noise}. Moreover, we have
		\begin{equation*}
		\mathcal{W}_{C_T,p}\left (L^N(X^{(N)}), \mathcal{L}(\bar X\vert B)\right ) \to 0,
		\quad \bar{\mathbb{P}}-\mbox{a.s.}
		\quad \mbox{as } N\to \infty.
		\end{equation*} 
	\end{corollary}
	
	\begin{proof}
		Since $B$ is independent from the other variables, we can assume, without loss of generality, that our probability space is of the form $(\bar\Omega, \bar{\mathcal{A}}, \bar{\mathbb{P}}) := (\Omega\times\Omega^\prime, \mathcal{A}\otimes\mathcal{A}^\prime, \mathbb{P}\otimes\mathbb{P}^\prime)$, that the random variables $\zeta^{i}, \bar{\zeta}, W^{i}$ and $\bar{W}$ are defined on a space $(\Omega, \mathcal{A}, \mathbb{P})$ and the random variable $B$ is defined on the space $(\Omega^\prime, \mathcal{A}^\prime, \mathbb{P}^\prime)$.
		
		For a fixed path $\beta \in C_T$, we consider the modified inputs, on $(\Omega,\mathcal{A},\mathbb{P})$, $W^{i, \beta} := W^{i} + \beta$ and $\bar{W}^\beta := \bar{W} + \beta$. Let $X^{(N),\beta}$ (respectively $X^\beta$) be the solution to equation \eqref{classical system} (resp. equation \eqref{limit mckean vlasov}) with input $(\zeta^{(N)}, W^{(N), \beta})$ (resp. $\bar{\zeta}, \bar{W}^\beta)$ given by Theorem \ref{classical mean field result}. The Lipschitz bound in Theorem \ref{classical mean field result} and the independence of $\zeta^i$ and $W^{i, \beta}$, via Lemma \ref{appendix:iid into wasserstein}, imply that, for $\mathbb{P}$-a.e.~$\omega$,		
		\begin{equation*}
		\mathcal{W}_{C_T, p}(L^N(X^{(N), \beta}(\omega)), \mathcal{L}(X^\beta)) \to 0.
		\end{equation*}
		
		Now we build the solution $\bar{X}$ and $X^{(N)}$ resp.~to \eqref{mckean vlasov common noise} and to \eqref{common noise particles}. We claim that the maps
		\begin{align*}
		\Omega\times C_T \ni (\omega,\beta) \mapsto X^\beta(\omega) \in C_T, \quad \Omega\times C_T \ni (\omega,\beta) \mapsto X^{(N),\beta}(\omega)
		\end{align*}
		have versions that are jointly measurable and, for such versions, we define $\bar{X}(\omega, \omega^\prime) = X^{B(\omega')}(\omega)$ and $X^{(N)}(\omega, \omega^\prime) = X^{(N),B(\omega')}(\omega)$. Note that, by the definition of $X^B$, for every fixed $\omega^\prime \in \Omega^\prime$, we have $\mathbb{P}$-a.s.
		\begin{equation*}
		dX^{B(\omega^\prime)} = b(t, X^{B(\omega^\prime)}, \mathcal{L}_{\mathbb{P}}(X^{B(\omega^\prime)}))dt + dW_t + dB_t(\omega^\prime),
		\end{equation*}
		where the law is taken with respect to the space $(\Omega, \mathcal{A}, \mathbb{P})$. But the independence of $B$ from the other variables implies that, $\bar{\mathbb{P}}$-a.s.,
		\begin{equation*}
		\mathcal{L}_{\mathbb{P}}(X^{B}) = \mathcal{L}_{\mathbb{P}\otimes \mathbb{P}^\prime}(X^B \vert B).
		\end{equation*}
		Hence $\bar{X}$ is a solution to equation \eqref{mckean vlasov common noise} on the product space $\Omega \times \Omega^\prime$. Similarly $X^{(N)}$ is a solution to \eqref{common noise particles} on $\Omega \times \Omega^\prime$. Therefore we have, for $\bar{\mathbb{P}}$-a.e. $(\omega, \omega^\prime)$,
		\begin{equation*}
		\mathcal{W}_{C_T,p}\left (L^N(X^{(N)})(\omega, \omega^\prime), \mathcal{L}(\bar X\vert B)(\omega^\prime) \right ) 
		= \mathcal{W}_{C_T, p}(L^N(X^{(N), \beta})(\omega), \mathcal{L}(X^\beta)) \mid_{\beta=B(\omega^\prime)} \to 0,
		\end{equation*}
		which is the desired convergence.
		
		It remains to prove the measurability claim on $X^\beta$ and $X^{(N),\beta}$. We prove it for $X^{(N),\beta}$, the proof for $\bar{X}$ being analogous. Recall the notation in Section \ref{section:main result} and note that the following maps are Borel measurable
		\begin{align*}
		&F_1:\mathcal{P}_p(C_T)\times \mathbb{R}^d \times C_T \ni (\mu,x_0,\gamma)\mapsto S^\mu(x_0,\gamma) \in C_T,\\
		&F_2:\mathcal{P}_p(\mathbb{R}^d \times C_T) \times C_T \ni (\nu,\beta) \mapsto (\cdot+(0,\beta))_\# \nu \in \mathcal{P}_p(C_T),
		\end{align*}
		(where $\cdot+(0,\beta)$ is the map on $\mathbb{R}^d \times C_T$ defined by $(x,\gamma)+(0,\beta)=(x,\gamma+\beta)$). Indeed, $F_1$ is continuous (because the solution of \eqref{linear equation} depends continuously on the drift, the initial data and the signal), $F_2$ is also Lipschitz-continuous (indeed, for any $(\beta,\nu)$ and $(\beta',\nu')$, if $m$ is an optimal plan between $\nu$ and $\nu'$, then $((\cdot+(0,\beta),\cdot+(0,\beta'))_\#m$ is an admissible plan between $F_2(\beta,\nu)$ and $F_2(\beta',\nu')$ and standard bounds give the Lipschitz property). 
		Moreover let $\Psi$ the map defined in \eqref{eq:definition psi}. It is continuous, hence measurable.
		Now we can write, for every $\beta$ in $C_T$, for every $i=1,\ldots N$,
		\begin{align*}
		X^{(N),\beta,i}(\omega) = F_1(\Psi(F_2(L^N(\zeta^{(N)}(\omega),W^{(N)}(\omega)),\beta)),\zeta^i(\omega),W^i(\omega)+\beta), \quad \mathbb{P}-\mbox{a.s.}
		\end{align*}
		and the right-hand side above is composition of measurable maps, hence measurable. Therefore the right-hand side is a measurable version of $X^{(N),\beta}$. The proof is complete.
	\end{proof}
	\subsection{Heterogeneous mean field}
	\label{section:heterogeneous mean field}
	
	As a further application of Theorem \ref{classical mean field result} we want to consider the case of heterogeneous mean field. We will show the convergence even when the drivers are not identically distributed. 
	This applies in particular to the results of the physical system studied in \cite{guhlke2018stochastic} as was discussed in the introduction. In that model, it is assumed that the state of each particle is influenced by its radius. Particle $i$ has a radius $r^i$, which is deterministic, and it is known that the radii are distributed according to a distribution $\lambda$. We allow here for the radii to be stochastic and not necessarily identically distributed, but still independent. Moreover, we will assume the volume to change in time.
	
	Heterogeneous mean field systems appear also in other contexts, see for example (among many others) \cite{Tou2014}, \cite{ChoKlu2015}, which work with semimartingale inputs and use a coupling \`a la Sznitman \cite{sznitman1991topics}.
	
	On the probability space $(\Omega, \mathcal{A}, \mathbb{P})$, we consider a family $(\zeta^{(N)}, W^{(N)}) = (\zeta^{i}, W^{i})_{i\geq 1}\subset L^p(\Omega, \mathbb{R}^d \times C_T(\mathbb{R}^d))$. This family is taken i.i.d.
	
	In addition, for each $N\in\mathbb{N}$, we consider a family $R^{(N)} = (R^{i,N})_{1 \leq i \leq N} \subset L^p(C_T(\mathbb{R}^n)^N)$.
	
	We construct the following interacting particle system
	\begin{equation}\label{eq:heterogeneous system}
	\left\{
	\begin{array}{l}
	dX^{i,N}_t = b(t, X^{i,N}_t, R_t^{i,N}, L^N(X^{(N)}_t,  R_t^{(N)})) dt + dW^{i}_t\\
	X^{i,N}_0 = \zeta^{i}.
	\end{array}
	\right.	
	\end{equation}
	We call this an heterogeneous particle system because the particles are not exchangeable anymore, if the $R^{i,N}$ are not exchangeable.
	
	We assume that the $R^{i,N}$ are independent of the $\zeta^{i}$ and $W^{i}$ and that there exists a measure $\lambda \in \mathcal{P}_p(C_T(\mathbb{R}^n))$ such that
	\begin{equation*}
	L^N(R^{(N)})(\omega) \overset{\ast}{\rightharpoonup} \lambda, \quad \mathbb{P}-\mbox{a.s.}
	\end{equation*}
	and actually in $p^\prime$-Wasserstein distance for $p^\prime>p$. We also consider the following mean field equation (on a probability space $(\Omega,\mathcal{A},\mathbb{P})$):
	\begin{equation}\label{eq:heterogeneous_McKVla}
	\left\{
	\begin{array}{l}
	d\bar X_t = b(t, \bar X_t, \bar R_t, \mathcal{L}(\bar X_t,  \bar R_t)) dt + d\bar W_t\\
	\bar X_0 = \bar \zeta
	\end{array}
	\right.	
	\end{equation}
	where $\bar \zeta$, $\bar W$ and $\bar R$ are independent random variables distributed resp.~as $\zeta^i$, $W^i$ and $\lambda$.
	The following result is a corollary of Theorem \ref{classical mean field result}. We also use Lemma \ref{lemma:modified lln} and Lemma \ref{lemma:convergence not iid inputs} to deal with the convergence of the input data.
	
	\begin{corollary}
		\label{cor: heterogeneus particles}
		Let $p\in[1, \infty)$, $p^\prime\in (p,\infty)$. Assume that $b: [0,T] \times \mathbb{R}^{d+n} \times \mathcal{P}_p(\mathbb{R}^{d+n}) \to \mathbb{R}^d$ is a measurable function and there exists a constant $K_b$ such that, 
		\begin{equation*}
		\vert b(t, x, \mu) - b(t, x^\prime, \mu^\prime) \vert^p
		\leq K_b \left(\vert x - x^\prime \vert^p  
		+ \mathcal{W}_{\mathbb{R}^{d+n},p}(\mu, \mu^\prime)^p\right),
		\end{equation*}
		$\forall t \in [0,T], x, x^\prime \in \mathbb{R}^{d+n}, \mu, \mu^\prime \in \mathcal{P}_p(\mathbb{R}^{d+n})$.
		
		Let $(\Omega, \mathcal{A}, \mathbb{P})$ be a probability space. On this space we consider independent families $\zeta^{(N)} = (\zeta^{i})_{i\geq1}\subset L^{p^\prime}(\Omega, \mathbb{R}^d)$, $W^{(N)} = (W^{i})_{i\geq1} \in L^{p^\prime}(\Omega, C_T)$ of i.i.d.~random variables. Let $\bar \zeta$ be distributed as $\zeta^1$ and let $\bar W$ be distributed as $W^1$ and independent of $\bar \zeta$. Moreover, assume that $R^{(N)} = (R^{i,N})_{1\leq i \leq N}$ is a family of independent random variables in $L^{p^\prime}(\Omega, \mathbb{R}^n)$ which are independent from the others.
		If there is convergence of the heterogeneous part (in $p^\prime$-Wasserstein distance),
		\begin{equation*}
		\mathcal{W}_{C_T(\mathbb{R}^{n}), p^\prime}(L^N(R^{(N)}), \mathcal{L}(\bar{R})) \to 0
		\quad \mathbb{P}-\mbox{a.s.}
		\quad \mbox{as } N\to \infty,
		\end{equation*}
		then also the solution converges (in $p$-Wasserstein distance),
		\begin{equation*}
		\mathcal{W}_{C_T(\mathbb{R}^{d+n}), p} (L^N(X^{(N)}, R^{(N)}), \mathcal{L}(\bar{X}, \bar{R}))
		\quad \mathbb{P}-\mbox{a.s.}
		\quad \mbox{as } N\to \infty,
		\end{equation*} 
	\end{corollary}
	
	\begin{proof}
		We start by rewriting the system \eqref{eq:heterogeneous system} so that we can invoke Theorem \ref{classical mean field result}. We change the state space of the system from $\mathbb{R}^d$ to $\mathbb{R}^d \times \mathbb{R}^n$ and we define on this new space the process $Y_t^{i,N} := (X_t^{i, N}, R_t^{i,N})$. Clearly, $X^{i,N}$ is a solution to system \eqref{eq:heterogeneous system} if and only if $Y^{i, N}$ solves
		\begin{equation}\label{eq:heterogeneous_system_enlarged}
		\left\{
		\begin{array}{l}
		dY^{i,N}_t = \left(\begin{array}{c}
		b(t, Y^{i,N}_t, L^N(Y^{(N)}_t)) \\
		0
		\end{array}\right)dt 
		+ d\left(\begin{array}{c}
		W^{i}_t\\
		R^{i,N}_t
		\end{array}\right)\\
		Y^{i,N}_0 = \left(\begin{array}{c}
		\zeta^{i}\\
		R^{i,N}_0
		\end{array}\right).
		\end{array}
		\right.	
		\end{equation}
		A similar transformation can be applied to the McKean-Vlasov equation to obtain that $\bar Y_t = (\bar X_t, \bar R_t)$ solves
		\begin{equation*}
		\left\{
		\begin{array}{l}
		d\bar Y_t = \left(\begin{array}{c}
		b(t, \bar Y_t, \mathcal{L}(\bar Y_t)) \\
		0
		\end{array}\right)dt 
		+ d\left(\begin{array}{c}
		\bar W_t\\
		\bar R_t
		\end{array}\right)\\
		\bar Y_0 = \left(\begin{array}{c}
		\bar \zeta\\
		\bar R_0
		\end{array}\right).
		\end{array}
		\right.	
		\end{equation*}
		In this setting the inputs satisfy the assumption of Theorem \ref{classical mean field result}. Hence, we obtain the following inequality. $\forall \omega \in \Omega$,
		\begin{align*}
		\mathcal{W}_{C_T(\mathbb{R}^{d+n}), p} & (L^N(X^{(N)}, R^{(N)}), \mathcal{L}(\bar{X}, \bar{R}))^p \\
		& \leq C \mathcal{W}_{\mathbb{R}^{d} \times C_T(\mathbb{R}^{d+n}), p}(L^N(\zeta^{(N)}, R^{(N)}, W^{(N)}), \mathcal{L}(\bar{\zeta}, \bar{R}, \bar{W}))^p.
		\end{align*}
		By Lemma \ref{lemma:convergence not iid inputs} (with $X_i := (\zeta^{i,N}, W^{i,N})$ and $Y_{i,N} := (R^{i,N})$ on the spaces $E := \mathbb{R}^d \times C_T$ and $F := \mathbb{R}^n$), $L^N(\zeta^{(N)}, R^{(N)}, W^{(N)})$ converges weakly to $\mathcal{L}(\bar{\zeta}, \bar{R}, \bar{W})$ $\mathbb{P}$-a.s.. Now, for every $q$ with $p<q<p^\prime$, $L^N(\zeta^{(N)}, W^{(N)}$ converges in $q$-Wasserstein distance, $\mathbb{P}$-a.s., by Lemma \ref{appendix:iid into wasserstein} and $L^N(R^{(N)})$ converges also in $q$-Wasserstein distance, $\mathbb{P}$-a.s., by assumption. In particular, $\mathbb{P}$-a.s., $L^N(\zeta^{(N)}, R^{(N)}, W^{(N)})$ have uniformly (in $N$) bounded $q$-th moments. Hence, by Lemma \ref{lem:weak_implies_W}, $L^N(\zeta^{(N)}, R^{(N)}, W^{(N)})$ converges also in $p$-Wasserstein distance, $\mathbb{P}$-a.s., and so $\mathcal{W}_{C_T(\mathbb{R}^{d+n}), p} (L^N(X^{(N)}, R^{(N)}), \mathcal{L}(\bar{X}, \bar{R}))$ tends to $0$. The proof is complete.
	\end{proof}

The following variant of the strong law of large numbers will be useful to prove Lemma \ref{lemma:convergence not iid inputs}.
\begin{lemma}\label{lemma:modified lln}
	Let $(X_i)_{i\geq 1}$ be a sequence of i.i.d.~real-valued centered random variables and let $(Y_{i,N})_{1 \leq i \leq N}$ be an independent family of real-valued independent random variables. Moreover, assume that there exists $C>0$ such that
	\begin{equation*}
	\Vert X_i \Vert_{L^4(\mathbb{R})} \leq C,
	\quad \Vert Y_{i,N} \Vert_{L^4(\mathbb{R})} \leq C,
	\quad \forall i,N \geq 1.
	\end{equation*}
	Then,
	\begin{equation*}
	S^N := \frac1N \sum_{i=1}^{N}X_iY_{i,N} \to 0,
	\quad \mathbb{P}-\mbox{a.s.}
	\end{equation*}
\end{lemma}

\begin{proof}
	We first establish a bound on the fourth moment of the empirical sum $S^N$.
	\begin{align*}
	\mathbb{E}\vert S^N \vert^4
	= & \frac{1}{N^4}  \sum_{i=1}^{N}\mathbb{E}\left[ X_i^4 \right]\mathbb{E}\left[ Y_{i,N}^4 \right]
	+ \frac{6}{N^4} \sum_{i,j=1}^{N}\mathbb{E}\left[ X_i^2 \right]\mathbb{E}\left[ X_j^2 \right]\mathbb{E}\left[ Y_{i,N}^2 \right]\mathbb{E}\left[ Y_{j,N}^2 \right] \leq \frac{C}{N^2}.
	\end{align*}
	Only those two terms in the sum do not vanish, because the $X_i$'s are centered. The constant $C$ depends on the upper bounds of the random variables. Let $p<\frac{1}{4}$,
	\begin{equation*}
	E_N := \left\{ \vert S^N \vert > \frac{1}{N^p} \right\}.
	\end{equation*}
	Using Chebychev inequality, we have the following
	\begin{equation*}
	\sum_{N=1}^\infty \mathbb{P}\{E^N\} 
	\leq \sum_{N=1}^\infty N^{4p}\mathbb{E}[S^N] 
	\leq C \sum_{N=1}^\infty N^{4p - 2}.
	\end{equation*}
	For our choice of $p$, we have convergence of the series. Borel Cantelli's Lemma implies that 
	\begin{equation*}
	\mathbb{P}\{  \limsup_{N\to\infty} E^N\} =0,
	\end{equation*}
	which in turn implies almost sure convergence of $S^N$.
\end{proof}
	
		\begin{lemma}\label{lemma:convergence not iid inputs}
		Let $p \in [1, \infty)$ be fixed.
		Let $(X_i)_{i\geq 1}$ be a sequence of i.i.d.~random variables on a space $(\Omega, \mathcal{A}, \mathbb{P})$ taking values in a Polish space $E$, with law $\mu \in \mathcal{P}_p(E)$. Let $(Y_{i,N})_{1\leq i\leq N}$ be another sequence of random variables taking values on a Polish space $F$, which is independent from $(X_i)_{i\geq 1}$.
		Assume that there exists a probability measure $\lambda \in \mathcal{P}_p(F)$ such that
		\begin{equation}\label{convergence of the ys}
		L^N(Y^{(N)}) := \frac{1}{N}\sum_{i=1}^{N}\delta_{Y_{i,N}} 
		\overset{\ast}{\rightharpoonup} \lambda,
		\quad \mathbb{P}-\mbox{a.s.}
		\end{equation}
		Then,
		\begin{equation*}
		L^N(X^{(N)}, Y^{(N)}) \overset{\ast}{\rightharpoonup} \mu \otimes \lambda,
		\quad \mathbb{P}-\mbox{a.s.}
		\end{equation*}
	\end{lemma}
	
	\begin{proof}
		Since $(X_i)_{i\geq 1}$ are a sequence of i.i.d.~random variables, there exists a set of full measure $\Omega^x \subset \Omega$, such that $L^N(X^{(N)}(\omega)) \overset{\ast}{\rightharpoonup} \mu$, for every $\omega \in \Omega^x$. Weak convergence implies tightness of the sequence $(L^N(X^{(N)})(\omega))$, thus, for every $\epsilon > 0$, there exists a compact set $E^{\omega}_\epsilon \subset E$, such that
		\begin{equation*}
		L^N(X^{(N)}(\omega))((E^{\omega}_\epsilon)^c ) < \frac{\epsilon}{2},
		\quad \omega \in \Omega^x.
		\end{equation*}
		In a similar way, there exists a set of full measure $\Omega^y \subset \Omega$ such that for every $\epsilon > 0$ there exists a compact $F^{\omega}_\epsilon \subset F$ that satisfies		$L^N(Y^{(N)}(\omega))((F^{\omega}_\epsilon)^c ) < \frac{\epsilon}{2}$, $\omega \in \Omega^y$.
		For every $\omega \in \Omega^x \cap \Omega^y$, we can consider the compact $K^{\omega}_\epsilon = E^{\omega}_\epsilon \times F^{\omega}_\epsilon \subset E\times F$ and compute the following
		\begin{equation*}
		L^N(X^{(N)}(\omega), Y^{(N)}(\omega))((K^{\omega}_\epsilon)^c) 
		\leq L^N(X^{(N)}(\omega))((E^{\omega}_\epsilon)^c ) + L^N(Y^{(N)}(\omega))((F^{\omega}_\epsilon)^c )
		< \epsilon.
		\end{equation*}
		We have thus shown that the sequence $L^N(X^{(N)}, Y^{(N)})$ is almost surely tight. With an abuse of notation, we call $L^N$ a converging subsequence and we take a continuous and bounded test function of the form $\varphi(x, y) := \varphi_1(x)\varphi_2(y)$ on $E \times F$. We compute the following
		\begin{align*}
		L^N(X^{(N)}, Y^{(N)})(\varphi) - (\mu \otimes \lambda)(\varphi) 
		= & \frac{1}{N} \sum_{i=1}^{N}\varphi_2(Y_{i,N})\left[ \varphi_1(X_i) - \int_E \varphi_1(x)d\mu(x) \right] \\
		& + \frac{1}{N} \sum_{i=1}^{N}\int_E \varphi_1(x)d\mu(x)\left[ \varphi_2(Y_{i,N}) - \int_F \varphi_2(y)d\lambda(y) \right].
		\end{align*}
		The first term on the right hand side converges to zero thanks to Lemma \ref{lemma:modified lln}, since the term in the brackets is a collection of bounded centered i.i.d.~random variables. The second term on the right-hand side converges by assumption \eqref{convergence of the ys}.
	\end{proof}

\begin{remark}\label{rmk:heterogeneous_noise}
The same result Corollary \ref{cor: heterogeneus particles} holds actually in a slightly different context of heterogeneous noises, namely when, in equation \eqref{eq:heterogeneous system}, the noise $dW^i_t$ in equation is replaced by $d[ \sigma(R^i_t)W^i_t]$ and, in equation \eqref{eq:heterogeneous_McKVla}, the noise $d\bar{W}$ is replaced by $d[\sigma(\bar{R}^i_t)\bar{W}^i_t]$, for a continuous bounded function $\sigma:\R\to \R$. Indeed, one can repeat the proof of Corollary \ref{cor: heterogeneus particles} replacing the noise in equation \eqref{eq:heterogeneous_system_enlarged} by
\begin{align*}
d\left(\begin{array}{c}
\sigma(R^{i,N}_t)W^{i}_t\\
R^{i,N}_t
\end{array}\right)
\end{align*}
and similarly for corresponding McKean-Vlasov SDE, and one gets
\begin{align*}
\mathcal{W}_{C_T(\mathbb{R}^{d+n}), p} & (L^N(X^{(N)}, R^{(N)}), \mathcal{L}(\bar{X}, \bar{R}))^p \\
& \leq C \mathcal{W}_{\mathbb{R}^{d} \times C_T(\mathbb{R}^{d+n}), p}(L^N(\zeta^{(N)}, R^{(N)}, [\sigma(R)W]^{(N)}), \mathcal{L}(\bar{\zeta}, \bar{R}, [\sigma(\bar{R})\bar{W}]))^p.
\end{align*}
Then one notes that $(\sigma(R^{i,N})W^{i},R^{i,N})$ is a continuous function of $(W^{i},R^{i,N})$, hence, since $L^N(\zeta^{(N)}, R^{(N)}, W^{(N)})$ converges weakly $\mathbb{P}$-a.s., then also $L^N(\zeta^{(N)}, R^{(N)}, [\sigma(R)W]^{(N)})$ converges weakly $\mathbb{P}$-a.s.. Moreover, the convergence in $q$-Wasserstein distance, for $p<q<p^\prime$, of $L^N(\zeta^{(N)}, W^{(N)})$ and of $L^N(R^{(N)})$ and the boundedness of $\sigma$ imply the $\mathbb{P}$-a.s. uniform (in $N$) bound on the $q$-th moments of $L^N(\zeta^{(N)}, R^{(N)}, [\sigma(R)W]^{(N)})$. Hence $L^N(\zeta^{(N)}, R^{(N)}, [\sigma(R)W]^{(N)})$ converges also in $p$-Wasserstein distance $\mathbb{P}$-a.s..
\end{remark}

\section{Large Deviations}
\label{section:large deviations}

In this section we assume that the driving paths $W$ of equation \eqref{mckean vlasov equation} live on the space $C_{T,0}$ of continuous functions starting at $0$. The results of Sections \ref{section:main result} and \ref{section:applications} apply also in this case.

Let $p \in [1, \infty)$. Let $b : [0,T] \times \mathbb{R}^d \times \mathcal{P}_p(\mathbb{R}^d) \to \mathbb{R}^d$ be a drift as before and such that it satisfies \ref{assumption}.

As in Section \ref{section:main result}, we define the function
\begin{equation*}
\begin{array}{cccc}
\Phi: &\mathcal{P}_p(\mathbb{R}^d \times C_{T, 0})\times \mathcal{P}_p(C_T) & \rightarrow & \mathcal{P}_p(C_T)\\
& (\mathcal{L}(\zeta, W), \mu) & \mapsto & \mathcal{L}(X^\mu) = (S^\mu)_\# \mathcal{L}(\zeta, W),
\end{array}
\end{equation*}
where $S^\mu$ is the solution map of ODE \eqref{eq:ode}, as defined in \eqref{defn:solution map}, with $\mathbb{R}^d \times C_{T,0}$ instead of $\mathbb{R}^d \times C_T$ as a domain. Similarly, we consider the map $\Psi$ defined as in \eqref{eq:definition psi}, replacing $C_{T}$ with $C_{T,0}$.

We introduce, for every $\mu$ in $\mathcal{P}_p(C_T)$, the map
\begin{equation}
\label{eq:definion f mu}
f^\mu : C_T \ni \gamma \mapsto \left(\gamma_0, \gamma_\cdot - \gamma_0 - \int_{0}^{\cdot}b(s, \gamma_s, \mu_s) ds\right) \in \mathbb{R}^d \times C_{T,0}.
\end{equation}
Note that $f^\mu = (S^\mu)^{-1}$ and $f^\mu$ is continuous, in particular measurable.

\begin{lemma}
	\label{ldp:joint law}
	Let $T>0$ be fixed and let $p \in [1,\infty)$, assume \ref{assumption}. The function $\Psi$ is a bijection, with inverse given by $\Psi^{-1}(\mu) = f^\mu_\# \mu$.
\end{lemma}

\begin{proof}
	For every $\nu$ in $\mathcal{P}_p(\mathbb{R}^d\times C_T)$ and $\eta$ in $\mathcal{P}_p(C_T)$, we have
	\begin{equation*}
	\Phi(\nu,\mu) = (S^\mu)_\# \nu = \eta \text{ if and only if } \nu = f^\mu_\# \eta.
	\end{equation*}
	In particular, with $\eta=\mu$, we get that $\Psi(\nu)=\mu$ if and only if $\nu = f^\mu_\# \mu$. Hence $\Psi$ is invertible, with inverse given by $\Psi^{-1}(\mu) = f^\mu_\# \mu$ (one can also show that $\Psi^{-1}$ is continuous).
	%
	%
\end{proof}

For $N\in \mathbb{N}$, let $(\zeta^{(N)}, W^{(N)}) = (\zeta^{i,N}, W^{i,N})_{1\leq i\leq N}: \Omega \to (\mathbb{R}^d  \times C_{T, 0})^N$ be a family of random variables. We consider the system of interacting particles on $\mathbb{R}^{d}$ as defined in \eqref{classical system}, namely
\begin{equation}\label{ldp:particles}
\left\{
\begin{array}{l}
dX^{i,N} = b(t, X^{i,N}, L^N(X^{(N)})) dt + dW^{i,N}_t\\
X^{i, N}_0 = \zeta^{i,N}.
\end{array}
\right.
\end{equation}
with solution $X^{(N)} := (X^{i,N})_{i=1,\cdots, N}$. We have seen in Section \ref{section:classical mean field limit} that we can define a suitable probability space $(\Omega_N, \mathcal{A}_N, \mathbb{P}_N)$, such that
\begin{equation*}
\mathcal{L}_{\mathbb{P}_N}(\zeta^{(N)}, W^{(N)}) 
= L^N(\zeta^{(N)}, W^{(N)}) 
:= \frac{1}{N} \sum_{i=1}^N \delta_{(\zeta^{i,N}, W^{i,N})},
\end{equation*}
and equation \eqref{mckean vlasov equation} is exactly the interacting particle system \eqref{ldp:particles}. Let $(\bar \zeta, \bar W) \in L^p(\mathbb{R}^d  \times C_{T, 0})$, we call $\bar X \in L^p(C_T)$ the solution to the related McKean-Vlasov equation \eqref{limit mckean vlasov}.

This construction shows that $\Psi$ is a continuous function that maps the empirical measure of the inputs into the empirical measure of the particles, namely
\begin{equation*}
\Psi\left(L^N(\zeta^{(N)}, W^{(N)})\right) = L^N(X^{(N)}), \quad \forall N\in \mathbb{N}.
\end{equation*}
This suggests the following immediate application to the contraction principle for large deviations. 
\begin{lemma}
	\label{ldp:contraction principle}
	Let $(\zeta^{(N)}, W^{(N)}) = (\zeta^{i,N}, W^{i,N})_{1\leq i\leq N} \subset L^p(\mathbb{R}^d  \times C_{T, 0})$ be a sequence of random variables and let $I: \mathcal{P}_p(\mathbb{R}^d \times C_{T,0}) \to [0, +\infty]$ be a lower semi-continuous function. Assume that that $L^N(\zeta^{(N)}, W^{(N)})$ satisfies a large deviations principle with (good) rate function $I$, in the sense of Definition \ref{def: large deviations}.
	
	Let $X^{(N)} = (X^{i,N})_{i=1,\dots, N}$ be the solution to the interacting particle system \eqref{ldp:particles} with inputs $(\zeta^{i,N}, W^{i,N})_{i=1,\dots, N}$. Then the empirical law $L^N(X^{(N)})$ satisfies a large deviations principle with (good) rate function
	\begin{equation*}
	J(\mu) := I(\Psi^{-1}(\mu)) = I(f^\mu_{\#}\mu), \quad \forall \mu\in\mathcal{P}_p(C_T).
	\end{equation*}
\end{lemma}	

\begin{proof}
	We know that the function $\Psi$ is a continuous function, we can thus apply the contraction principle for large deviations which ensures that $L^N(X^{(N)})$ satisfies a large deviations principle with rate function
	\begin{equation*}
	J(\mu) := \inf \left\{
	I(\nu) \mid \forall \nu \in \mathcal{P}_p(\mathbb{R}^d\times C_{T,0}), \quad \Psi(\nu) = \mu
	\right\},
	\quad \mu \in \mathcal{P}_p(C_{T}).
	\end{equation*}
	From the bijectivity of $\Psi$, given by Lemma \ref{ldp:joint law}, we deduce that
	\begin{equation*}
	J(\mu) = I(\Psi^{-1}(\mu)) = I(f^\mu_{\#}\mu), \quad \mu \in \mathcal{P}_p(C_{T}).
	\end{equation*}
\end{proof}

Given a Polish space $E$, the relative entropy between two measures $\mu, \mu^\prime \in \mathcal{P}_p(E)$ is defined as
\begin{equation*}
H(\mu \mid \mu^\prime) := 
\left\{
\begin{array}{ll}
\int_{E} \log(\frac{d\mu}{d\mu^\prime}) d\mu, & \mu << \mu^\prime,\\
+\infty, & otherwise.
\end{array}
\right.
\end{equation*}
We can specialize Lemma \ref{ldp:contraction principle} to the case when the rate function of the inputs is the entropy with respect to a specific measure. In this case we obtain an even more explicit rate function for the convergence of the empirical measure of the particles.
\begin{lemma}
	\label{ldp:large deviations}
	Let $(\zeta^{(N)}, W^{(N)}) = (\zeta^{i,N}, W^{i,N})_{1\leq i\leq N} :\Omega \to (\mathbb{R}^d  \times C_{T, 0})^N$ be a sequence of random variables such that: There exists $\bar \nu \in \mathcal{P}_p(\mathbb{R}^d \times C_{T,0})$ such that $L^N(\zeta^{N)}, W^{(N)})$ satisfies a large deviations principle with good rate function
	\begin{equation*}
	H(\nu \mid \bar \nu), \quad \forall \nu \in \mathcal{P}_p(\mathbb{R}^d \times C_{T,0}).
	\end{equation*}
	
	Let $X^{(N)} = (X^{i,N})_{i=1,\dots, N}$ be the solution to the interacting particle system \eqref{ldp:particles} with inputs $(\zeta^{i,N}, W^{i,N})_{i=1,\cdots, N}$. Then the empirical law $L^N(X^{(N)})$ satisfies a large deviations principle with good rate function
	\begin{equation*}
	H(\mu \mid \Phi(\bar \nu, \mu)), \quad \forall \mu\in\mathcal{P}_p(C_T).
	\end{equation*}
\end{lemma}	

\begin{proof}
	
	We can apply Lemma \ref{ldp:contraction principle} to obtain that $L^N(X^{(N)})$ satisfies a large deviations principle with rate function
	\begin{equation*}
	I(\mu) := H(\Psi^{-1}(\mu) \mid \bar \nu), \quad \mu \in \mathcal{P}_p(C_{T}).
	\end{equation*}
	We show now that $H(\Psi^{-1}(\mu) \mid \bar \nu) = H(\mu \mid \Phi(\bar \nu, \mu))$. For this, note that, by Lemma \ref{ldp:joint law} and by the definition of $\Phi$,
	\begin{equation*}
	\Psi^{-1}(\mu) = f^\mu_{\#}\mu, \quad \bar \nu = f^\mu_{\#}\Phi(\bar \nu, \mu).
	\end{equation*}
	Here $f^\mu_{\#}$ is a push-forward via a measurable map $f^\mu$ with measurable inverse $S^\mu$. Hence, by standard facts in measure theory, $\Psi^{-1}(\mu) \ll \bar\nu$ if and only if $\mu \ll \Phi(\bar\nu,\mu)$, in which case we have
	\begin{equation*}
	\frac{d\Psi^{-1}(\mu)}{d\bar\nu} =  \frac{d\mu}{d\Phi(\bar\nu,\mu)} \circ S^\mu.
	\end{equation*}
	Hence, in the case that $\Psi^{-1}(\mu)$ is not absolutely continuous with respect to $\bar \nu$, we have $H(\Psi^{-1}(\mu) \mid \bar \nu) = H(\mu \mid \Phi(\bar \nu, \mu)) = +\infty$. In the case that $\Psi^{-1}(\mu)$ is absolutely continuous with respect to $\bar\nu$, we have
	\begin{equation*}
	H(\Psi^{-1}(\mu) \mid \bar \nu) 
	= \int \frac{d\Psi^{-1}(\mu)}{d\bar\nu} \log \frac{d\Psi^{-1}(\mu)}{d\bar\nu} d\bar\nu
	= \int \frac{d\mu}{d\Phi(\bar\nu,\mu)} d(S^\mu_\#\bar\nu)
	= H( \mu\mid\Phi(\bar \nu, \mu)).
	\end{equation*}
	The proof is complete.
	%
\end{proof}
We will now apply Sanov's Theorem to i.i.d. inputs. The case when the convergence happens in the Wasserstein metric was proved in \cite{MR2593592}, and it requires an exponential integrability assumption on the law of the inputs.
\begin{theorem}
	\label{ldp:large deviations iid inputs}
	Let $(\zeta^{i}, W^{i})_{i\geq 1} \subset L^p(\mathbb{R}^d  \times C_{T, 0})$ be a sequence of i.i.d.~random variables with law $\bar \nu := \mathcal{L}(\zeta^{1}, W^{1})$.
	Assume that there exists $(x^0, \gamma^0) \in \mathbb{R}^d \times C_{T,0}$ such that
	\begin{equation*}
	\log \int_{\mathbb{R}^d \times C_{T,0}} \exp(\lambda (\vert x - x^0 \vert + \Vert \gamma - \gamma^0\Vert_{\infty})^p) d\bar \nu(x, \gamma) < + \infty, 
	\quad \forall \lambda>0.
	\end{equation*}
	Let $X^{(N)} := (X^{i,N})_{i=1,\dots, N}$ be the solution to the interacting particle system \eqref{ldp:particles} with inputs $(\zeta^{(N)}, W^{(N)}) := (\zeta^{i}, W^{i})_{i=1,\cdots, N}$. Then the empirical law $L^N(X^{(N)})$ satisfies a large deviations principle with good rate function
	\begin{equation*}
	H(\mu \mid \Phi(\bar \nu, \mu)), \quad \forall \mu\in\mathcal{P}_p(C_T).
	\end{equation*}
\end{theorem}	

\begin{proof}
	Sanov's theorem, as in \cite[Theorem~1.1]{MR2593592}, gives that the empirical measure $L^N(\zeta^{(N)}, W^{(N)})$ satisfies a large deviations principle with good rate function
	\begin{equation*}
	I(\nu) = H(\nu \mid \bar \nu), \quad \forall \nu \in \mathcal{P}_p(\mathbb{R}^d\times C_{T,0}).
	\end{equation*}
	
	The proof then follows from Lemma \ref{ldp:large deviations}.
\end{proof}

\section{Central limit theorem}
\label{sec: central limit theorem}

In this section we study the fluctuations of the empirical measure around the limit. In order to do so, we apply an abstract result of Tanaka. In its original paper \cite{MR780770}, Tanaka studied McKean-Vlasov stochastic differential equations with linear drift, here we show how this can be also applied to more general drifts. 

We restate now \cite[Theorem 1.1]{MR780770}. Let $E$ be a Polish space and $\mathcal{M}(E)$ (resp. $\mathcal{P}(E)$) the space of signed (resp. probability) measures on $E$. In this section, given a function $f(x)$ on $E$, we use the notation $f(\mu)$ to denote $\int_E f(x) d\mu(x)$, for $\mu \in \mathcal{M}(E)$.

%
%
%

\begin{theorem}[Tanaka]
	\label{thm: tanaka clt}
	Let $f: E \times \mathcal{P}(E) \to \R$ be a bounded function such that there exists
	\begin{equation*}
	f^{\prime}: E \times E \times \mathcal{P}(E) \to \R
	\end{equation*}
	such that
	\begin{enumerate}[label=(\roman*), ref=\ref{thm: tanaka clt} (\roman*)]
		\item 
		\label{asm: differentiable function: boundedness} 
		$f^{\prime}$ is bounded.
		
		\item 
		\label{asm: differentiable function: strong continuity}
		There exists a constant $C>0$ such that, for all $\mu,\nu \in \mathcal{P}_p(E)$,
		$$\sup_{x,y \in E}| f^{\prime}(x,y,\mu) - f^{\prime}(x,y,\nu) | \leq \mathcal{W}_{p, E}(\mu, \nu ).$$
		
%
%
		
		\item 
		\label{asm: differentiable function: algebraic relation}
		for all $x\in E$, $\mu, \nu \in \mathcal{P}(E)$, 
		$$
		f(x, \nu) - f(x, \mu) = \int_{0}^{1} f^{\prime}(x,\nu - \mu, \mu + \theta [\nu - \mu])  d \theta,
		$$
		where we used the notation $f^{\prime}(x,\rho,\mu) = \int_{E}f^{\prime}(x,y,\mu) \rho(dy)$, for $\rho \in \mathcal{M}$.
	\end{enumerate}
	
	Assume that $(X^{i})_{i\in \N}$ is a sequence of independent and identically distributed random variables on $E$ with distribution $\mu$. We define
	$$
	\mu^N := L^N(X^{(N)}) := \frac{1}{N} \sum_{i=1}^{N} \delta_{X^i},
	$$
	$$
	Y^N := \sqrt{N} [ f ( \mu^N , \mu^N ) - f ( \mu, \mu )],
	$$
	where we used the notation $f(\nu,\mu) = \int_{E}f(y,\mu) \nu(dy)$.
	
	Then, the probability distribution of $Y^N$ converges to a Gaussian distribution with mean $0$ and variance $\sigma^2$, where
	$$
	\sigma^2 = \int_{E} [ f(x,\mu) +  f^{\prime}(x, \mu, \mu) - m ]^2 \mu(dx)
	$$
	$$
	m = \int_{E} [ f ( x , \mu ) + f ^ {\prime} (x, \mu, \mu )] \mu(dx).
	$$
\end{theorem}

\begin{remark}
	We changed slightly the conditions, the proof of the Theorem is exactly the same in this case as in \cite{MR780770}.
\end{remark}

The main idea behind Theorem \ref{thm: tanaka clt} is that one needs to linearize the solution map with respect to the measure. Hence, we introduce the following definition of differentiability with respect to a probability measure.

\begin{definition}
	\label{def: measure differentiable}
	Let $E$ be a Polish space. A function $b: E \times \mathcal{P}_p(E)  \to \R ^ d$ is said to have a linear functional derivative if there exists a function:
	\begin{equation*}
	\partial_{\mu} b: E  \times E \times \mathcal{P}_p(E) \ni ( x , y , \mu ) \to \partial_{\mu} b ( x, y,  \mu ) \in \R ^ d,
	\end{equation*}
	continuous for the product topology, such that, for any $x\in E$ and any bounded subset $\mathcal{K} \subset \mathcal{P}_p(E)$, the function $y \to  \partial_{\mu}b( x , y , \mu )$ is at most of $p$-growth in $y$, uniformly in $\mu \in \mathcal{K}$, and 
	\begin{equation*}
	b( x , \mu ^ \prime ) - b( x , \mu) = \int_{ 0 }^{ 1 } \partial_{ \mu } b ( x,  \mu ^ \prime - \mu , \mu + \theta [ \mu ^ \prime - \mu ] ) d \theta, 
	\qquad 
	\forall x \in \R ^ d, 
	\;
	\mu , \mu^\prime \in \mathcal{P}_p(\R^d).
	\end{equation*}
\end{definition}

We prove the central limit theorem under suitable differentiability assumptions on the drift with respect to the measure argument. In this section we assume the following assumption.
\begin{assumption}
	\label{asm: drift clt}
	Let $b : \mathcal{P}_p(\R^d) \times \R^d \to \R^d$ and $K > 0$, assume 
	\begin{enumerate}[label=(\roman*), ref= \ref{asm: drift clt} (\roman*)]
		\item \label{asm: drift clt: space diff}
		$b$ differentiable in the spatial variable $x$ with derivative $\partial_{x}b$.
		
		\item \label{asm: drift clt: measure diff}
		$b$ differentiable in the sense of Definition \ref{def: measure differentiable}, with derivative $\partial_{\mu}b$.
		
		\item \label{asm: drift clt: space measure diff}
		$\partial_{\mu} b$ differentiable in the spatial variable $y$ with derivative $\partial_{y}\partial_{ \mu } b $.
		
		\item \label{asm: drift clt: Lipschitz}
		(uniform Lipschitz continuity) For all $ x, x^\prime, y, y^{\prime} \in \mathbb{R}^d, \mu, \mu^\prime \in \mathcal{P}_p(\mathbb{R}^d)$,
		$$
		\vert b(\mu, x) - b(\mu^\prime, x^\prime) \vert
		\leq K \left( \mathcal{W}_{\mathbb{R}^d, p}(\mu, \mu^\prime) + \vert x - x^\prime \vert
		\right),
		$$
		$$
		\vert \partial_{x} b(\mu, x) - \partial_{x}b(\mu^\prime, x^\prime) \vert
		\leq K \left( \mathcal{W}_{\mathbb{R}^d, p}(\mu, \mu^\prime) + \vert x - x^\prime \vert
		\right),
		$$
		$$
		\vert \partial_{\mu} b(\mu, x, y) - \partial_{\mu}b(\mu^\prime, x^\prime, y^\prime) \vert
		\leq K \left( \mathcal{W}_{\mathbb{R}^d, p}(\mu, \mu^\prime) 
		+ \vert x - x^\prime \vert
		+ \vert y - y^\prime \vert
		\right),
		$$
		$$
		\vert \partial_{y} \partial_{\mu} b(\mu, x, y) - \partial_{y} \partial_{\mu}b(\mu^\prime, x^\prime, y^\prime) \vert
		\leq K \left( \mathcal{W}_{\mathbb{R}^d, p}(\mu, \mu^\prime)
		+  \vert x - x^\prime \vert
		+ \vert y - y^\prime \vert
		\right).
		$$

		\item \label{asm: drift clt: boundedness}
		(uniform boundedness) For all $ x, y \in \mathbb{R}^d, \mu \in \mathcal{P}_p(\mathbb{R}^d)$.
		\begin{equation*}
		\vert b(x, \mu)\vert, 
		\vert \partial_{x}b(x, \mu)\vert, 
		\vert \partial_{\mu} b(\mu, x, y)\vert,
		\vert \partial_{y} \partial_{\mu} b(\mu, x, y)\vert
		\leq K.
		\end{equation*}
		
	\end{enumerate}
\end{assumption}
\begin{remark}
	Let $f \in C^1(\R^d)$, and $g\in C^1_b( \R^d \times \R^d \times \R^d; \R^d)$, then 
	$$
	b(x, \mu) := f \left ( g(x, \mu, \mu) \right) = f \left ( \int_{\R^d\times \R^d} g(x, y, z) \mu(dy)\mu(dz) \right)
	$$
	satisfies Assumption \ref{asm: drift clt}. The standard, linear case is when $f(x) = x$ and $g(x,y,z) = g( x , y )$.	
\end{remark}

To significantly simplify the notation in this section, we assume without loss of generality that all the particles start at $0$ and we remove the dependence of the solution on the initial condition.
With the previous simplification, we have that,	for $\mu \in C_T$, the solution map defined in \eqref{defn:solution map} is $S(\mu, \gamma) = S^\mu : C_T \to C_T$. Moreover, for $p \in [1, \infty)$,  $\Psi : \mathcal{P}_p(C_T) \to \mathcal{P}_p(C_T)$ is the fixed point map defined in \eqref{eq:definition psi}. We first look at the derivative of $F(\gamma ,\nu) := S(\Psi(\nu), \gamma)$ with respect to $\nu$, denoted $F^{\prime}(\gamma, \bar \gamma, \nu)$. 
For $f \in B = C_b(C_T; \R^d)$, define
\begin{equation*}
(A_t(\nu)f) (\gamma) := \partial_{x} b ( F _t (\gamma , \nu ), \Psi ( \nu ) _t ) \; f (\gamma)
+ \int_{C_T} \partial_y \partial_{\mu} b (S ( F _t (\gamma , \nu ), F _t(\tilde \gamma , \nu ), \Psi( \nu ) _t) \; f( \tilde \gamma)  \; d \nu ( \tilde \gamma ) ,
\end{equation*}
\begin{equation*}
G_t(\gamma, \bar \gamma, \nu) := \partial_{\mu} b ( F _t (\gamma , \nu ) , F _t ( \bar \gamma , \nu ) , \Psi ( \nu ) _t ).
\end{equation*}
The derivative $F^{\prime}$ formally satisfies the following linear differential equation in the Banach space $B$, with parameters $\gamma \in C_T$ and $\nu \in \mathcal{P}_p(C_T)$,
\begin{equation}
\label{eq: measure derivative}
\frac{d}{dt} F^{\prime}_t(\cdot, \bar \gamma, \nu) = A_t(\nu) F^{\prime} _t (\cdot, \bar \gamma, \nu) + G _t ( \cdot , \bar \gamma , \nu ) ,
\quad F^{\prime}_t |_{ t = 0 } = 0,
\end{equation}	
It follows from Assumption \ref{asm: drift clt} that the linear operator $A$ and the forcing term $G$ are bounded, uniformly in $t, \gamma, \nu$. 
\begin{lemma}
	\label{lem: properties measure derivative}
	Assume that $b$ satisfies Assumption \ref{asm: drift clt}. Then, for every $\gamma \in C_T$ and $\nu \in \mathcal{P}(C_T)$, equation \eqref{eq: measure derivative} admits a unique solution $F^{\prime}$. Moreover,
	\begin{enumerate}[label=(\roman*), ref= \ref{lem: properties measure derivative} (\roman*)]
		\item \label{lem: properties measure derivative: boundedness}
		$ \| F^{\prime}_t(\gamma, \nu) \|_B \leq C(K) $, for all $\gamma \in C_T$, $\nu \in \mathcal{P}_p(C_T)$, $t \in [ 0, T ]$.
		
		\item \label{lem: properties measure derivative: lipschitz}
		$ \| F^{\prime}_t(\gamma, \mu) - F^{\prime}_t(\gamma, \nu)\|_B \leq C(K) \mathcal{W}_{p, C_T}(\mu, \nu) $, for $\gamma \in C_T$, $\nu, \mu \in \mathcal{P}_p(C_T)$, $t \in [ 0, T ]$.
	\end{enumerate}
\end{lemma}

\begin{proof}
	Since $A$ is a bounded linear operator, we know from standard theory of ordinary differential equation that equation \ref{eq: measure derivative} admits a unique solution $F^{\prime}$  that satisfies
	\begin{equation*}
	\| F^{\prime}_t \|_{B} \leq \| G_t \|_{B} e^{\| A_T \|_{\mathcal{L}(B;B)}},
	\qquad
	t \in [0,T].
	\end{equation*}
	The proof of \ref{lem: properties measure derivative: boundedness} and \ref{lem: properties measure derivative: lipschitz} follows now form Assumption \ref{asm: drift clt}.
\end{proof}

It is now left to verify that the derivative $F^{\prime}$ of $F = S(\Psi)$ satisfies equation \eqref{eq: measure derivative}. We first need the following properties of the solution map.
\begin{lemma}
	\label{lem: properties convergence}
	Let $\nu,\nu^\prime \in \mathcal{P}(C_T)$. For $\epsilon \in [0,1]$, we define $\nu^{\epsilon} = \nu + \epsilon [\nu^{\prime} - \nu]$. We have the following,
	
	\begin{enumerate}[label=(\roman*), ref= \ref{lem: properties convergence} (\roman*)]
		\item \label{lem: properties convergence: wasserstein}
		$\mathcal{W}_p(\nu^{\epsilon}, \nu) \to 0$, as $\epsilon \to 0$.
		\item \label{lem: properties convergence: wasserstein image}
		$\mathcal{W}_p(\mu^{\epsilon}, \mu) = \mathcal{W}_p(\Psi(\nu^{\epsilon}), \Psi(\nu)) \to 0$, as $\epsilon \to 0$.
		\item \label{lem: properties convergence: uniform}
		$\sup_{\gamma \in C_T} \| S(\mu^\epsilon,\gamma)  - S(\mu,\gamma) \|_{C_T} \to 0$, as $\epsilon \to 0$.
	\end{enumerate}
	
\end{lemma}

\begin{proof}
	\ref{lem: properties convergence: wasserstein} follows from the tightness of $\nu, \nu^\prime$ and \ref{equivalence Wasserstein weak convergence on metric spaces}.
	
	\ref{lem: properties convergence: wasserstein image} follows from \ref{lem: properties convergence: wasserstein} and the Lipschitz continuity of $\Psi$, \ref{item 2 of main}.
	
	\ref{lem: properties convergence: uniform} is implied by \ref{lem: properties convergence: wasserstein image} and straight-forward computations.
\end{proof}

\begin{lemma}
	\label{lem: formal derivative}
	For every $ \gamma \in C_T$, the function $ \mathcal{P}_p(C_T) : \nu \to F _t( \gamma , \nu ) = S ( \Psi (\nu), \gamma)$ is differentiable in the sense of Definition \ref{def: measure differentiable} and its derivative satisfies equation \eqref{eq: measure derivative}.
\end{lemma}	

\begin{proof}
	Let $\nu, \nu^\prime \in \mathcal{P}_p(C_T)$ and let $F^{\prime}(\gamma, \bar \gamma, \nu)$ be a solution to equation \eqref{eq: measure derivative}. Using the equations for $S$, Lemma \ref{lem: properties convergence} and standard (but lengthy) computations it can be proved that
	\begin{equation*}
	\lim_{\epsilon \to 0} \frac{S(\Psi(\nu^{\epsilon}), \gamma) - S(\Psi(\nu), \gamma)}{\epsilon} -  \int_{C_T} F^{\prime}_t(\gamma, \bar \gamma, \nu) d [\nu^\prime - \nu](\bar \gamma) = 0.
	\end{equation*}
\end{proof}

The main result of this section is the following, which is a corollary of Theorem \ref{thm: tanaka clt}.
\begin{corollary}
	\label{cor: central limit thm}
	Let $(W^i)_{i\in \N}$ be a family of independent and identically distributed random variables on the Banach space $C_T$ with law $\nu$ and let $X^{(N)} = (X^{i,N})_{i=1,\dots, N}$ be the solution of the interacting particle system \eqref{intro:interacting particles} with input $(W^i)_{i\in \N}$. Let $W$ be a random variable on $C_T$ with law $\nu$, we call $X$ the solution to the McKean-Vlasov equation \eqref{intro:mckean vlasov equation} driven by $W$. Define $\mu := \mathcal{L}(X) = \Psi(\nu)$.
	
	Let $\varphi : C_T \to \R$ be a bounded Fr\'echet-differentiable test function with bounded derivative $\varphi^{\prime}$. We have that
	\begin{equation*}
	Y^N := \sqrt{N} \left( \frac{1}{N} \sum_{i=1}^{N} \varphi( X^{i,N} ) - \mu(\varphi) \right)
	\end{equation*}
	converges, as $N\to \infty$, to a Gaussian $N(0, \sigma ( \varphi ) )$ with
	$$
	\sigma^2(\varphi) = \int_{C_T} [ \varphi( F (\gamma , \nu ) ) + \int_{ C_T } \varphi ^ \prime( F (\bar \gamma , \nu )) F^{\prime}(\nu, \bar \gamma, \gamma) \nu ( d \bar \gamma ) - m ]^2 \nu ( d \gamma ),
	$$
	$$
	m(\varphi) = \int_{ C_T } [  \varphi ( F (\gamma , \nu ) ) 
	+  \int_{ C_T } \varphi ^ \prime( F ( \bar \gamma , \nu ) ) F^{\prime} ( \nu, \bar \gamma, \gamma ) \nu ( d \bar \gamma )
	]  \nu ( d \gamma ),
	$$		
	where $F = S(\Psi)$ and $F^{\prime}$ is the solution to equation \eqref{eq: measure derivative}.
\end{corollary}

\begin{proof}
	The function $f(\gamma, \nu) := \varphi(S(\Psi(\nu), \gamma))$ satisfies the assumption of Theorem \ref{thm: tanaka clt} with derivative $f^{\prime}(\gamma, \bar \gamma, \nu) = \varphi^{\prime}(S(\Psi(\nu), \gamma)) F(\gamma, \bar \gamma, \nu)$, where $\varphi^{\prime}$ is the Fr\'echet-derivative of $\varphi$ and $F$ is a solution to equation \eqref{eq: measure derivative}.
	
	Assumptions \ref{asm: differentiable function: boundedness} and \ref{asm: differentiable function: strong continuity} follow from Lemma \ref{lem: properties convergence}. Assumption \ref{asm: differentiable function: algebraic relation} follows from Lemma \ref{lem: formal derivative}.
\end{proof}

\section{Reflection at the boundary}
\label{section:boundary conditions}

The problem of SDEs in a domain with reflection has been considered since the works by Skorokhod \cite{Sko1961}, \cite{Sko1962}. The literature is vast and we mention the works by Tanaka \cite{Tan1979}, Lions and Sznitman \cite{LioSzn1984} as two of the most important papers. The case of mean field SDEs with reflection has also been studied, see for example the works by Sznitman \cite{Szn1984}, Graham and Metivier \cite{GraMet1989}, which establish well-posedness under general conditions and particle approximation for independent inputs and with Brownian motion as driving signal (possibly with a diffusion coefficient). Also other types of SDEs with mean field interactions and in domains have been studied (with different kind of reflections), see for example \cite{HamLed2017}, \cite{BCGL2016}.

Here we show how to adapt the main result, Theorem \ref{main theorem}, and the argument to the case of reflecting boundary conditions. With respect to the previously cited works, we can allow general continuous paths as inputs, we do not need to assume independence nor exchengeability of particles for particle approximation.

Throughout this section, we assume that $D$ is a bounded convex polyhedron in $\mathbb{R}^d$ with nonempty interior (see Remark \ref{rmk:extensions_D} below for extensions). 

We are given a Borel vector field $b$ that satisfies the following
	\begin{assumption}
	\label{assumption refrlection}
	Let $p\in[1, +\infty)$. The function $b:[0,T]\times\bar{D}\times\mathcal{P}_p(\bar{D}\times\mathbb{R}^d)\rightarrow\mathbb{R}^d$ is a measurable function and there exists a constant $K_b$ such that, 
	\begin{equation*}
	\vert b(t, x, \mu) - b(t, x^\prime, \mu^\prime) \vert^p
	\leq K_b \left(\vert x - x^\prime \vert^p  
	+ \mathcal{W}_{\mathbb{R}^d,p}(\mu, \mu^\prime)^p\right),
	\end{equation*}
	$\forall t \in [0,T], x, x^\prime \in \mathbb{R}^d, \mu, \mu^\prime \in \mathcal{P}_p(\bar{D}\times\mathbb{R}^d)$.
\end{assumption}

 We consider the generalized McKean-Vlasov Skorokhod problem
\begin{equation}
\label{eq:McKVlaSkor}
\left\{
\begin{array}{l}
dX_t = b(t,X_t,\mathcal{L}(X_t,k_t)) dt + dW_t -dk_t\\
X \in C_T(\bar{D}),\ \ X_0=\zeta,\\
k \in BV_T,\ \ d|k|_t = 1_{X_t\in \partial D} d|k|_t,\ \ dk_t = n(X_t) d|k|_t.
\end{array}
\right.
\end{equation}
Let $(\Omega, \mathcal{A}, \mathbb{P})$ be a probability space, the input to equation \ref{eq:McKVlaSkor} is a random variable $(\zeta,W)$ with values in $\bar{D}\times C_T$, the solution is the couple $(X,k)$ of random variables satisfying the equation above, $|k|$ denotes the total variation process of $k$ (not the modulus of $k$) and $n(x)$ is the outer normal at $x$, for $x$ in $\partial D$, see Remark \ref{rmk:boundary_corner} below for the precise meaning. A short explanation on the meaning of the $k$ term is given later after the main result.

We give now the precise definition of solution:
\begin{definition}
	Let $(\Omega,\mathcal{A},\mathbb{P})$ be a probability space and let $\zeta:\Omega\rightarrow D$, $W:\Omega\rightarrow C_T$ be random variables on it. A solution to the generalized McKean-Vlasov Skorokhod problem with input $(\zeta,W)$ is a couple of random variables $X:\Omega\rightarrow C_T(\bar{D})$ and $k:\Omega\rightarrow C_T$ such that, for Lebesgue-a.e.~$t$, $\mathcal{L}(X_t,k_t)$ is in $\mathcal{P}_p(\bar{D}\times \mathbb{R}^d)$ and, for a.e.~$\omega$, equation \eqref{eq:McKVlaSkor} is satisfied (where $X\in C_T(\bar{D})$ means that $X$ is $C_T(\bar{D})$-valued, $k\in BV_T$ means that $k\in BV_T:=BV([0,T];\mathbb{R}^d)$ $\mathbb{P}$-a.s.~and where the last line is understood in the sense of Remark \ref{rmk:boundary_corner} below).
\end{definition}

\begin{remark}\label{rmk:boundary_corner}
	Actually the last condition is only valid for smooth domains, which is not the case for $D$ convex polyhedron (it is not smooth at the intersections of the faces of the polyhedron). For simplicity of notation, here and in what follows (also for the particle system), we keep the formulation above, with the understanding that the precise condition should be: for a.e.~$\omega$ there exists a Borel function $\gamma=\gamma^\omega:[0,T]\rightarrow\mathbb{R}^d$ such that $dk_t=\gamma_t d|k|_t$ and, for $d|k|$-a.e.~$t$, $\gamma_t$ belongs to $d(X_t)$, where
	\begin{align*}
	d(x)=\left\{\sum_{i, x\in \partial D_i} \alpha_i n_i \mid \alpha_i\ge 0, \left|\sum_{i, x\in \partial D_i} \alpha_i n_i\right|=1\right\}
	\end{align*}
	and where $\partial D_i$ are the faces of the polyhedron with outer normals $n_i$.
\end{remark}

Our main result is, as before, well-posedness of the generalized McKean-Vlasov Skorokhod problem and Lipschitz continuity with respect to law of the input.

\begin{theorem}
	\label{thm:wellpos_cont_boundary}
	Let $T>0$ be fixed and let $p \in [1,\infty)$. Assume that $b$ satisfies \ref{assumption refrlection}.
	\begin{enumerate}
		\item For every input $(\zeta,W)$ (random variable in $L^p(\bar{D}\times C_T)$) with finite $p$-moment, there exists a unique solution $(X,k)$ to the generalized McKean-Vlasov Skorokhod problem \eqref{eq:McKVlaSkor}.
		\item There exists a constant $\tilde C = \tilde C(p, T, b) > 0$ such that: for every two inputs $(\zeta^i,W^i)$, $i=1,2$ (defined possibly on different probability spaces) with finite $p$-moments, the following is satisfied
		\begin{align*}
		W_{C_T(\bar{D})\times C_T, p}(\mathcal{L}(X^1, k^1), \mathcal{L}(X^2, k^2)) 
		\leq \tilde C W_{\bar{D} \times C_T,p}(\mathcal{L}(\zeta^1, W^1), \mathcal{L}(\zeta^2, W^2)).
		\end{align*}
		In particular, the law of a solution $(X,k)$ depends only on the law of $(\zeta,W)$.
	\end{enumerate}
\end{theorem}


To prove this result, we regard the generalized McKean-Vlasov Skorokhod problem as a fixed point problem with parameter. For this, we introduce the following Skorokhod problem, for fixed $\mu$ in $\mathcal{P}_p(C_T(\bar{D})\times C_T)$ (calling $\mu_t$ the marginal at time $t$):
\begin{equation}
 \label{eq:Skor_mu}
\left\{
\begin{array}{l}
dY^\mu_t = b(t,Y^\mu_t,\mu_t) dt + dW_t -dh^\mu_t\\
Y^\mu \in C_T(\bar{D}),\ \ Y^\mu_0=\zeta,\\
 h^\mu\in BV_T,\ \ d|h^\mu|_t = 1_{Y_t\in \partial D} d|h^\mu|_t,\ \ dh^\mu_t = n(Y_t) d|h^\mu|_t.
\end{array}
\right.
\end{equation}
We recall the following well-posedness result for $\mu$ fixed:
\begin{lemma}
	Fix $\mu$ in $\mathcal{P}_p(C_T(\bar{D})\times C_T)$ and assume that $b$ is Lipschitz and bounded as in Theorem \ref{thm:wellpos_cont_boundary}. Then, for every $T>0$, for every deterministic initial datum $\zeta\equiv x_0$ in $\bar{D}$ and for every deterministic path $W\equiv \gamma$ in $C_T$, there exists a unique solution $(Y,h)=(Y^\mu,h^\mu)$ in $C_T(\bar{D})\times C_T$ to the above equation.
\end{lemma}
This result is classical and one can see it as a consequence of well-posedness for Skorokhod problem without drift, via Lemma \ref{lem:DupIsh91} below, in the same line of the proof of Theorem \ref{thm:wellpos_cont_boundary} (see in particular the bound \eqref{eq:stability_mu}). We call $S^\mu:\bar{D}\times C_T \rightarrow C_T(\bar{D})\times C_T$ the solution map to \eqref{eq:Skor_mu}, that is, $S^\mu(x_0,\gamma)=(Y^\mu,h^\mu)$ where $(Y^\mu,h^\mu)$ solves \eqref{eq:Skor_mu} with deterministic input $(x_0,\gamma) \in \bar D \times C_T$.

For a general random input $(\zeta,W)$ in $L^p(\bar{D}\times C_T)$, this result, applied to $(\zeta(\omega),W(\omega))$ for a.e.~$\omega$, gives existence and pathwise uniqueness of a solution $(Y^\mu,h^\mu)$ to \eqref{eq:Skor_mu} and the representation formula $(Y^\mu,h^\mu)=S^\mu(\zeta,W)$. Moreover, again from Lemma \ref{lem:DupIsh91} below, if the input $(\zeta,W)$ has finite $p$-moment, then also the solution $(Y^\mu,h^\mu)$ has finite $p$-moment. We call
\begin{equation*}
\begin{array}{cccc}
\Phi:& \mathcal{P}_p(\bar{D}\times C_T)\times \mathcal{P}_p(C_T(\bar{D})\times C_T) &\to & \mathcal{P}_p(C_T(\bar{D})\times C_T),\\
& (\mathcal{L}(\zeta,W),\mu) & \mapsto & (S^\mu)_\# \mathcal{L}(\zeta,W),
\end{array}
\end{equation*}
the law of a probability measure $\mathcal{L}(\zeta,W)$, under the solution map $S^\mu_T$ of the Skorokhod problem with $\mu$ fixed. 

As in the case without boundaries, note that $(X,k)$ solves the McKean-Vlasov Skorokhod problem if and only if $\mathcal{L}(X,k)$ is a fixed point of $\Phi({\mathcal{L}(\zeta,W)}, \cdot)$. Hence, Theorem \ref{thm:wellpos_cont_boundary} reduces to a fixed point problem with parameter.

A key tool in the proof of this result is the Lipschitz dependence of the boundary term $k$ on the given path in the Skorokhod problem. The precise statement follows from \cite[Theorem 2.2]{MR1110990} (there the Skorokhod problem is formulated in the space of cadlag functions, but continuity of the solution is ensured by \cite[Lemma 2.4]{Tan1979}).

\begin{lemma}\label{lem:DupIsh91}
	Fix $T>0$. For $x_0$ in $\bar{D}$, $z$ in $C_T$. Then there exists a unique solution $(y,k)=(y^{x_0,z},k^{x_0,z})$ in $C_T(\bar{D})\times C_T$ to the Skorokhod problem driven by $z$, namely
	\begin{equation*}
	\left\{
	\begin{array}{l}
	d y = d z - d k,\\
	y \in C_T(D),\ \ y_0=x_0,\\
	k\in BV_T,\ \ d |k| = 1_{y\in \partial D} d|k|,\ \ dk = n(y) d|k|.
	\end{array}
	\right.
	\end{equation*}
	Moreover there exists $C\ge0$ (which is locally bounded in $T$) such that, for every $x_0^1$, $x_0^2$ in $D$, for every $z^1$, $z^2$ in $C_T$,
	\begin{align*}
	&\|y^{x_0^1,z^1}-y^{x_0^2,z^2}\|_{\infty} + \|k^{x_0^1,z^1}-k^{x_0^2,z^2}\|_{\infty}\le C|x_0^1-x_0^2| +C\|z^1-z^2\|_{\infty},\\
	&\|y^{x_0^1,z^1}-x_0^1\|_{\infty} + \|k^{x_0^1,z^1}\|_{\infty}\le C\|z^1\|_{\infty}.
	\end{align*}
\end{lemma}

\begin{proof}[Proof of Theorem \ref{thm:wellpos_cont_boundary}]
	The result follows from the abstract Proposition \ref{contraction theorem}, provided we verify conditions 1) and 2) on $\Phi$.
	
	Let $\mu \in \mathcal{P}_p(C_T(\bar{D})\times C_T)$ be fixed, let $\nu^1$ and $\nu^2$ be in $\mathcal{P}_p(\bar{D}\times C_T)$ and let $m$ be an optimal plan on $(\mathbb{R}^d\times C_T)^2$ for these two measures. On the probability space $((\bar{D}\times C_T)^2, m)$, we call $\zeta^i$, $W^i$, $i=1,2$, the r.v.~defined by the canonical projections and $(Y^i,h^i)=S^\mu(\zeta^i,W^i)$ the solution to the Skorokhod problem \eqref{eq:Skor_mu} with input $(\zeta^i, W^i)$. We have
	\begin{align*}
	W_p(\Phi(\nu^1,\mu),\Phi(\nu^2,\mu))^p \le \E^m (\|Y^1-Y^2\|_{\infty}+\|h^1-h^2\|_{\infty})^p,
	\end{align*}
	so it is enough to bound the right-hand side. We can apply Lemma \ref{lem:DupIsh91} to $z^i= \int^t_0 b(t, Y^i_r,\mu) d r +W^i$, $x_0^i=\zeta^i$ and so $y^i=Y^i$, $k^i=h^i$: we get
	\begin{align*}
	\|Y^1-Y^2\|_{\infty}+\|h^1-h^2\|_{\infty} \le C |\zeta^1-\zeta^2| +C \int^T_0 |b(t, Y^1_t,\mu)-b(t, Y^2_t,\mu)| d t + C\|W^1-W^2\|_{\infty}.
	\end{align*}
	Using the Lipschitz property of $b$ in $x$ (uniformly in $\mu$), we get
	\begin{align*}
	\|Y^1-Y^2\|_{\infty}+\|h^1-h^2\|_{\infty} \le C |\zeta^1-\zeta^2| +C \int^T_0 \|Y^1-Y^2\|_{\infty} d t + C\|W^1-W^2\|_{\infty}.
	\end{align*}
	By Gronwall inequality
	\begin{align*}
	\|Y^1-Y^2\|_{\infty}+\|h^1-h^2\|_{\infty} \le C |\zeta^1-\zeta^2| + C\|W^1-W^2\|_{\infty}.
	\end{align*}
	We take expectation (with respect to $m$) of the $p$-power and use the optimality of $m$, to obtain
	\begin{align*}
	W_p(\Phi(\nu^1,\mu),\Phi(\nu^1,\mu))^p \le C W_p(\nu^1,\nu^2)^p.
	\end{align*}
	This ends the proof of condition 1) of Proposition \ref{contraction theorem}.
	
	Let now $(\zeta, W)$ be fixed with law $\nu := \mathcal{L}(\zeta, W)$. Consider $\mu^1, \mu^2 \in \mathcal{P}_p(C_T(\bar{D})\times C_T)$ and call $(Y^i,h^i)=(Y^{\mu^i},h^{\mu^i})$, $i=1,2$ the corresponding solutions to the Skorokhod problem \eqref{eq:Skor_mu} (driven by the initial datum $\zeta$ and the path $W$). We can apply Lemma \ref{lem:DupIsh91} to $z^i= \int^t_0 b(t, Y^{\mu^i}_r,\mu^i) d r +W$, $x_0^i=\zeta$ and so $y^i=Y^{\mu^i}$, $k^i=h^i$: we get
	\begin{align*}
	\|Y^1-Y^2\|_{\infty} +\|h^1-h^2\|_{\infty} \le C \int^T_0 |b(t, X^{\mu^1}_r,\mu^1)-b(t, X^{\mu^2}_r,\mu^2)| d r.
	\end{align*}
	Taking the $p$-power and arguing as without boundaries, we get
	\begin{align*}
	\|Y^1-Y^2\|_{\infty}^p +\|h^1-h^2\|_{\infty}^p \le C \int^T_0 \|Y^1-Y^2\|_{\infty}^p d t + C\int^T_0 W_{C_t,p}(\mu^1,\mu^2)^p d t
	\end{align*}
	and so, by Gronwall inequality,
	\begin{align}
	\|Y^1-Y^2\|_{\infty}^p +\|h^1-h^2\|_{\infty}^p \le C\int^T_0 W_{C_t,p}(\mu^1,\mu^2)^p d t.\label{eq:stability_mu}
	\end{align}
	Taking expectation, we conclude
	\begin{align*}
	W_p(\Phi(\nu,\mu^1),\Phi(\nu,\mu^2))^p \le C\int^T_0 W_{C_t,p}(\mu^1,\mu^2)^p d t.
	\end{align*}
	As for without boundaries, iterating this inequality $k$ times for $k$ large enough (such that $(CT)^k/k!<1$), we get condition 2) in Proposition \ref{contraction theorem}. The proof is complete.
\end{proof}

As in the case without boundary, if the driving process is adapted, then so is the solution to the McKean-Vlasov Skorokhod problem. We omit the proof as it is completely analogous to the one without boundary.

\begin{proposition}
	Let $(\mathcal{F}_t)_t$ be a right-continuous, complete filtration on $(\Omega,\mathcal{A},\mathbb{P})$ such that $\zeta$ is $\mathcal{F}_0$-measurable and $W$ is $(\mathcal{F}_t)_t$-progressively measurable. Then the solution $(X,k)$ to \eqref{eq:McKVlaSkor} is also $(\mathcal{F}_t)_t$-progressively measurable.
\end{proposition}

Finally, following Section \ref{section:particle approximation}, we can obtain a particle approximation to the McKean-Vlasov Skorokhod problem \eqref{eq:McKVlaSkor}, just as corollary of the main result Theorem \ref{thm:wellpos_cont_boundary}. Here the corresponding particle system reads
\begin{equation}
\label{classical system Skor}
\left\{
\begin{array}{l}
dX^{i,N}_t = b(t,X^{i,N}_t,L^N(X^{(N)}_t,k^{(N)}_t)) dt + dW^{i,N}_t -dk^{i,N}_t\\
X^{i,N} \in C_T(\bar{D}),\ \ X_0^{i,N}=\zeta^{i,N},\\
k^{i,N}\in BV_T,\ \ d|k^{i,N}|_t = 1_{X^{i,N}_t\in \partial D} d|k^{i,N}|_t,\ \ dk^{i,N}_t = n(X^{i,N}_t) d|k^{i,N}|_t.
\end{array}
\right.
\end{equation}
Again the solution is an $N$-uple of couples $(X^{i,N},k^{i,N})_{i=1,\ldots N}$ (and again $|k^{i,N}|$ denotes the total variation process of $k^{i,N}$ and $k^{i,N}\in BV_T$ means that $k^{i,N}$ belongs to $BV_T$ $\mathbb{P}$-a.s.). The following result can be proven exactly as Theorem \ref{classical mean field result} (here we use a notation analogous to that theorem).

\begin{theorem}
	Let $p \in [1,\infty)$ and assume $b$ satisfies Assumption \ref{assumption refrlection}. Let $(\Omega, \mathcal{A}, \mathbb{P})$ be a probability space. On this space we consider, for $N \in \mathbb{N}$, a family of random variables $(\zeta^{(N)}, W^{(N)}) = (\zeta^{i,N}, W^{i,N})_{1\leq i\leq N}$ taking values on $\bar{D} \times C_T$. Let $\bar \zeta \in L^p(\Omega,  \bar{D})$ and $\bar W \in L^p(\Omega, C_T)$.
	Then:
	\begin{enumerate}[label=\roman{*} ]
		\item There exists a unique pathwise solution $(X^{(N)},k^{(N)})$ (resp. $(\bar X, \bar k)$) to the interacting particle system \eqref{classical system Skor} (resp. equation \eqref{eq:McKVlaSkor}).
		\item There exists a constant $C$ depending on $b$ such that for all $N \geq 1$, for a.e. $\omega \in \Omega$,
		\begin{align*}
		W_{C_T(\bar{D})\times C_T, p}&(L^N(X^{(N)}(\omega),k^{(N)}(\omega)), \mathcal{L}(\bar{X},\bar{k}))^p\\
		& \leq C W_{\bar{D} \times C_T, p}(L^N(\zeta^{(N)}(\omega), W^{(N)}(\omega)), \mathcal{L}(\bar{\zeta}, \bar{W}))^p.
		\end{align*}
		\item If the empirical $L^N(\zeta^{(N)}, W^{(N)})$ converges to $\mathcal{L}(\bar \zeta, \bar W)$ $\mathbb{P}$-a.s.,~then also the emprical measure of the solution converges.
	\end{enumerate}
\end{theorem}

\begin{remark}\label{rmk:extensions_D}
	More general cases can be treated, for example oblique reflection or even more general domains $D$, possibly with some extra assumptions: as one can see from the proof, it is enough to have an estimate as in Lemma \ref{lem:DupIsh91} for the boundary term. The case of oblique reflection (still with $D$ convex polyhedron) is treated in \cite{MR1110990} (see Assumptions 2.1 and Theorem 2.1 there). The case of more general domains is treated for example in \cite{Tan1979, MR873889}, though the Lipschitz constant in Lemma \ref{lem:DupIsh91} seems in this case to depend also on $z$.
\end{remark}

	\appendix
	\section{Proof of Proposition \ref{contraction theorem}}
	In this section we prove proposition \ref{contraction theorem}.
	
	First, we must show that $\Phi^Q$ has a unique fixed point. If $k = 1$, it is exactly the contraction principle, so we will assume $k>1$. Clearly $(\Phi^Q)^k$ is a contraction, hence it is has a unique fixed point $P_Q$. Hence,
	\begin{equation*}
	d_E(\Phi^Q(P_Q), P_Q) = d_E((\Phi^Q)^{k+1}(P_Q), (\Phi^Q)^k(P_Q)) \leq c d_E(\Phi^Q(P_Q), P_Q).
	\end{equation*}
	Since $c<1$, this implies $d_E(\Phi^Q(P_Q), P_Q) =0$ and therefore $P_Q$ is also a fixed point for $\Phi^Q$. Every fixed point of $\Phi^Q$ is also a fixed point for $(\Phi^Q)^k$, hence $P_Q$ is the only fixed point of $\Phi^Q$.
	
	We are left to prove \eqref{inequality}. By induction, one can show that
	\begin{equation*}
	\forall Q, Q^\prime \in F, \forall P \in E
	\quad d_E((\Phi^Q)^k(P), (\Phi^{Q^\prime})^k(P)) 
	\leq \left(\sum_{i=1}^k L^i\right) d_F(Q, Q^\prime).
	\end{equation*}
	Using a triangular inequality as well as assumption $2)$ and the previous inequality we obtain
	\begin{align*}
	d_E(P_Q, P_{Q^\prime})
	= & d_E((\Phi^Q)^k(P), (\Phi^{Q^\prime})^k(P^\prime))\\
	\leq & d_E((\Phi^Q)^k(P), (\Phi^{Q})^k(P^\prime))
	+ d_E(\Phi^k_{Q}(P^\prime), (\Phi^{Q^\prime})^k(P^\prime))\\
	\leq & cd_E(P_Q, P_{Q^\prime})
	+ \left(\sum_{i=1}^k L^i\right) d_F(Q, Q^\prime).
	\end{align*}
	The proof is complete.
	
	\section{Wasserstein Metric}
	
	We now recall some useful information on the Wasserstein metric, which we defined in \eqref{Wasserstein}. For more details the reader can refer to \cite{ambrosio2008gradient}. Let $p\in[1, \infty)$.
	\begin{enumerate}[label=\roman{*} ]
		
		\item The infimum in the definition of Wasserstein metric is a minimum. For each couple $\mu, \nu \in \mathcal{P}_p(E)$ there exists a measure $m\in \Gamma(\mu, \nu)$ such that
		\begin{equation}\label{minimum wasserstein}
		\mathcal{W}_{E,p}(\mu, \nu)^p = \iint_{E\times E} d(x, y)^p m(dx, dy).
		\end{equation}
		
		\item\label{marginal minor full space} The Wasserstein distance of two measures on the space of paths is larger than the distance of the corresponding one-time marginals at $t$, for any $t$. Indeed, note that, for any $\mu, \nu \in \mathcal{P}_p(C_T)$, if $m$ is in $\Gamma(\mu, \nu)$, then $m_t \in \Gamma(\mu_t, \nu_t)$, therefore we have
		\begin{align*}
		\mathcal{W}_{\mathbb{R}^d, p}(\mu_t, \nu_t)^p 
		\leq & \iint_{\mathbb{R}^d \times \mathbb{R}^d} \vert x - x^\prime \vert^p m_t(dx, dx^\prime)
		=  \iint_{C_T \times C_T} \vert \gamma_t - \gamma^\prime_t \vert^p m(d\gamma, d\gamma^\prime)
		\leq  \mathcal{W}_{C_T,p}(\mu, \nu)^p.
		\end{align*}

		\item\label{equivalence Wasserstein weak convergence on metric spaces}
		Let $E$ be a Polish space. For any given sequence $(\mu^n)_{n\geq 1} \in \mathcal{P}_p(E)$ the following are equivalent
		\begin{enumerate}
			\item (The sequence converges in Wassertein sense) $ \lim_{n\to\infty}\mathcal{W}_{E,p}(\mu_n, \mu) = 0$.
			\item (The sequence converges weakly and is uniformly integrable) There exists $x_0 \in E$ such that,
			\begin{equation*}
			\left\{
			\begin{array}{l}
			\mu_n \overset{\ast}{\rightharpoonup} \mu, \quad \mbox{as } n\to \infty\\
			\lim_{k\to\infty}\int_{E\setminus B_k(x_0)} d^p(x, x_0) d\mu^n(x) = 0, \quad \mbox{uniformly in }n.
			\end{array}
			\right.
			\end{equation*}
			Cf. \cite[Proposition~7.1.5]{ambrosio2008gradient}.
		\end{enumerate}
		
		
	\end{enumerate} 
	
	As a consequence of point \eqref{equivalence Wasserstein weak convergence on metric spaces}, we give a sufficient condition to pass from weak convergence of measures to convergence in the $p$-Wasserstein distance. 
	
	\begin{lemma}\label{lem:weak_implies_W}
		Let $(E,d)$ be a Polish space and $\mu_n$, $n\in \mathbb{N}$, $\mu$ be probability measures on $E$, fix $q\in[1,\infty)$. If the sequence $(\mu_n)_{n\in\mathbb{N}}$ converges to $\mu$ in the weak topology on probability measures and if, for some $p \in(q, \infty)$ and some $x_0$ in $E$,
		\begin{align}
		\sup_n \int_E d(x,x_0)^{p} \mu_n(d x) <\infty,\label{eq:unif_peps_bd}
		\end{align}
		then $\mu_n$ converges in $q$-Wasserstein metric to $\mu \in \mathcal{P}_q(E)$.
	\end{lemma}
	
	\begin{proof}
		By property \eqref{equivalence Wasserstein weak convergence on metric spaces}, it is enough to show that the map $x\mapsto d(x,x_0)^q$ is uniformly integrable with respect to $(\mu_n)_n$. For this, we have, for any $R>0$, for any $n$,
		\begin{align*}
		\int_{d(x,x_0)>R} d(x,x_0)^q \mu_n(d x) \le R^{p-q} \int_E d(x,x_0)^{p} \mu_n(d x).
		\end{align*}
		By the uniform bound \eqref{eq:unif_peps_bd}, we can choose $R$ large enough to make the right-hand side above small for all $n$. This shows that $x\mapsto d(x,x_0)^q$ is uniformly integrable.
	\end{proof}
	
	\begin{lemma}
		\label{appendix:iid into wasserstein}
		Given $p\in(1,\infty)$ and a separable Banach space $(E,\vert \cdot \vert)$, let $(X^i)_{i\geq1} \in L^p(\Omega, E)$ be a family of i.i.d.~random variables on this space with law $\mu$. Then,
		\begin{equation*}
		\lim_{N\to\infty}\mathcal{W}_{E, q}(L^N(X^{(N)}), \mu) = 0, \quad q \in (1, p), \quad \mathbb{P}-a.s.
		\end{equation*}
	\end{lemma}
	
	\begin{proof}
		Since  $(X^i)$ are i.i.d., $\mathbb{P}$-a.s.~convergence in the weak topology
		\begin{align*}
		L^N(X^{(N)}(\omega)) \overset{\ast}{\rightharpoonup} \mathcal{L}(X^1), \quad \mathbb{P}-a.s.
		\end{align*}
		is a classical result, see for example \cite{MR622172} and references therein. Moreover, by the law of large numbers, we have, for a.e.~$\omega$,
		\begin{align*}
		\int_{E} \vert x \vert^{p} dL^N(X^{(N)}(\omega))(x)
		= \frac{1}{N}\sum_{i=1}^N \vert X^i(\omega) \vert^{p} \rightarrow \mathbb{E}\vert X^1\vert ^{p}  <\infty.
		\end{align*}
		We obtain condition \eqref{eq:unif_peps_bd} in Lemma \ref{lem:weak_implies_W}, which concludes the proof.
	\end{proof}

\bibliographystyle{abbrv}
\bibliography{bibliography} 

\end{document}